\documentclass[english,aps,preprint,showpacs]{revtex4}
\usepackage[T1]{fontenc}
\usepackage[latin9]{inputenc}
\usepackage{amssymb}
\usepackage{amsmath}
\usepackage{esint}
\usepackage{graphicx}
\usepackage{color}
\makeatletter

\@ifundefined{textcolor}{}
{%
 \definecolor{BLACK}{gray}{0}
 \definecolor{WHITE}{gray}{1}
 \definecolor{RED}{rgb}{1,0,0}
 \definecolor{GREEN}{rgb}{0,1,0}
 \definecolor{BLUE}{rgb}{0,0,1}
 \definecolor{CYAN}{cmyk}{1,0,0,0}
 \definecolor{MAGENTA}{cmyk}{0,1,0,0}
 \definecolor{YELLOW}{cmyk}{0,0,1,0}
}

\newtheorem{theorem}{Theorem}[section]
\newtheorem{proposition}{Proposition}[section]
\newenvironment{proof}{{\noindent\it Proof} \quad}{\hfill $\square$ \par}
\newtheorem{lemma}{Lemma}[section]
\renewcommand{\Re}{\operatorname{Re}}
\renewcommand{\Im}{\operatorname{Im}}

\makeatother
\usepackage{babel}

\begin{document}

\title{The Eigenvalue Problem of Nonlinear Schr\"odinger Equation at Dirac Points of Honeycomb Lattice}

\author{Yejia Chen\footnote{sjtu-cyj@alumni.sjtu.edu.cn}}
\author{Ruihan Peng}
\author{Qidong Fu}
\author{Fangwei Ye}
\author{Weidong Luo\footnote{wdluo@sjtu.edu.cn}}

\affiliation{School of Physics and Astronomy, Shanghai Jiao Tong University, Shanghai 200240, China}
\date{\today}

\begin{abstract}
We give a rigorous deduction of the eigenvalue problem of the nonlinear Schr\"odinger equation (NLS) at Dirac Points for potential of honeycomb lattice symmetry. Based on a bootstrap method, we observe the bifurcation of the eigenfunctions into eight distinct modes from the two-dimensional degenerated eigenspace of the regressive linear Schr\"odinger equation. We give the existence, the way of construction, uniqueness in $H^2$ space and the $C^\infty$ continuity of these eigenfunctions.
\end{abstract}

\maketitle

\section{Introduction}

This article focuses on the nonlinear phenomenon of Honeycomb lattice, which attracts intense interests around the physics and mathematics societies. As a frequently investigated two-dimensional models, honeycomb lattice has been widely researched in many fields of science. In condensed matter physics, the electronic structure of the graphene is one of the most famous applications \cite{Graphene-review:RevModPhys.81.109,honeycomb-exp-kekule-PhysRevLett}. In the quantum optics, there are also some important applications in several photonic honeycomb crystal \cite{honey-NLS:EXP-Omri,Honey-NLS:optical-edgewave-PRA,SSB:bifurcation-expnature-phon,Honey-NLS:Photonic-self-localized-PhysRevLett.111.243905}, reporting the unusual properties of the honeycomb lattice like the self-localization and periodical tunnelling patterns. Among all of these important findings, one of the significant properties of the honeycomb lattice are the special dispersion relation of the linear Hamiltonian model $H_{L}=-\delta+V_{hc}(\textbf{x})$. The origin of the dispersion relation comes from the solid physics when considering the energy bands of the electronic structure. As the Floquet-Bloch theorem states, the eigenvalue problems $H_{L}u(\textbf{x})=Eu(\textbf{x})$ can be decomposed into the subproblem $H_{L}(\textbf{k},\textbf{x})u(\textbf{k},\textbf{x})=E(\textbf{k})u(\textbf{k},\textbf{x})$. Associated with the so-called Bloch state $u(\textbf{k},\textbf{x})=v(\textbf{k},\textbf{x})e^{i\textbf{k}\cdot \textbf{x}}$, where $v(\textbf{x})$ is a periodic function of the honeycomb lattice. The evolution of the eigenvalue $E(\textbf{k})$ with the parameter $\textbf{k}$ in the Fourier space exactly represents the energy dispersion relation with the momentum of the electron for an single electron effective model. The uniqueness of honeycomb lattice lies in the corners of hexagonal the Brillouin zone, defined as the closure of the points $\textbf{k}$ closer to the origin points than any other points of the dual lattice $\mathbb{Z}^{2}$, the parameters of the discrete Fourier transformation of the torus constructed by identifying the points $\textbf{x}$ related through honeycomb periodicity. The conical shape of the dispersion surface near the corners, also called Dirac points by physicists, shading light on the unusual electromagnetic properties of materials like graphene and so on. At the Dirac points, the theoretical analysis \cite{Math-Honeycomb:Fefferman} exhibit the degeneration of the eigenspace for certain eigenvalue and its conical singularities. Especially, for weak potential, the ground states are exactly two-dimensionally degeneration, which is consistent with the observations in physical researches \cite{Graphene-review:RevModPhys.81.109}. This degeneration and the singularities are closely related to the finite symmetry of honeycomb lattice, resulting in the sensitivity of physical structure described by the honeycomb models towards external environment \cite{honeycomb-exp-kekule-PhysRevLett} or internal distortion \cite{honey-NLS:Zhuyi-Pra-distorted} and nonlinear effects \cite{SSB:bifurcation-expnature-phon}.

Nonlinearity constantly causes novel phenomena like solitons in physical and mathematical practices \cite{Math-Nonlinear-system-Yang}. For Bose-Einstein condensate (BEC), some researches investigated the self-trapping and symmetry breaking of the system in a double-well potential \cite{NLS:Smerzi-original-PhysRevLett,NLS:Raghavan-original-PhysRevA,NL:Coullet-non-hermitian-PRE,NL:Coullet-non-hermitian-Journalphys-B}. Unlike the linear cases, additional nonlinear terms significantly change the characteristics of the system just due to the varying of the norm of the target wavefunction, say $L^2$ norm for example. For the eigenvalue problem of nonlinear Hamiltonian $H_{NL}=-\Delta + V_{NL}(\textbf{x},\phi)$, the dependence of the potential on the wavefunction implies the failure of linear operator in theory coping with the eigenvalue problem. Rigorously speaking, the eigenspace is even not well-defined, since the solution $H_{NL}\phi(\textbf{x})=E\phi(\textbf{x})$ for fixed $E$ never forms a linear space, contradicting to the usual definition. It is still valuable to investigate such eigenvalue problem, though. In the context of the nonlinear Schr\"odinger equations (NLS), the justification of the existences of the eigenfunction indicates the stationary states of the system, no matter stable or unstable, showing the experimental feasibility of these models. Plus, these special states also offer as strong hints for the classification of the phase trajectories. Some stable states are attractors while some unstable states remark the critical points of the classification (See \cite{NLS:Smerzi-original-PhysRevLett,NLS:Raghavan-original-PhysRevA,NL:Coullet-non-hermitian-PRE,NL:Coullet-non-hermitian-Journalphys-B}). More importantly, as the varying of norm, some stationary states may emergent while some others disappear, which is called the bifurcation problem of the stationary states. Bifurcation only belongs to nonlinear dynamics, since the wavefunction never change by multiplication with a constant in a linear system. It repeated occurs in physical researches as phenomena of spontaneous symmetry breaking \cite{bifurcation-NL:Rahmi-asym2wells}, in stark differences with some examples of eigenspace modification by external symmetry-breaking effects \cite{honeycomb-exp-kekule-PhysRevLett}.

For strong interests of both the periodical structure and the nonlinear effects, surveys on the nonlinear Schr\"odinger equations in periodic potential are a natural extension of the previous researches. For instance, Bifurcation problem of Gross-Pitaevskii equation for periodic potential of any dimension was discussed in \cite{Math-Bifurcation:periodical-nonlinear-Dohnal}. The soliton and breather solutions of NLS in an array of Bose-Einstein condensates were reported in \cite{NLS:1dimarray-PhysRevLett}. Multiple solutions in $L^2(\mathbb{R}^2)$ were found in the periodic nonlinear systems under specific conditions \cite{Math-NLS:DING-periodical-multisolution,Math-NLS:Periodical-homoclinic}. For the widely usage and interesting function of honeycomb-type structure, there were also researches specifically focused on the equations in honeycomb lattice: the linear and nonlinear travelling of the edge states \cite{Honey-NLS:optical-edgewave-PRA,Math-NLS:Edge-States-Weinstein-Zhuyi}, dynamics of nonlinear waves in the deformed honeycomb lattice \cite{honey-NLS:Zhuyi-Pra-distorted,honey-NLS:Zhuyi2}, conical diffraction of the tight-binding lattice \cite{Honey-NLS:Zhuyi3}, and results by considering the approximate Dirac equation \cite{Math-NLS:William-localized-equations}. In particular, we point out the dynamics of the wave pocket compositions of the two-dimensional degenerated ground states dominated by NLS in honeycomb lattice are reported in \cite{Math-NLS:nonlinear-diracEQ-Jack}. In that paper, they also take use of the featured eigenfunctions of the Dirac points, which coincides with this research. However, we focus on periodical function of which the norms are defined in the torus, but not $\mathbb{R}^2$. For weak self-interactions and external potential, the assumption of Bloch periodicity is reasonable, and we give more detailed calculation and conclusions of the features of the eigenfunctions.

The goal of this article is to solve the eigenvalue problem of the nonlinear Schr\"odinger equation $H_{NL}\phi(\textbf{x})=E\phi(\textbf{x})$ for $H^1$ functions of small norm at Dirac points in the torus corresponding to weak, smooth honeycomb potential, where $H^s$ means the functions of the Sobolev space which have up to $s$ order weak derivative function in $L^2$ space ($s \in Z^{+}$). Our main result is Theorem \ref{theorem-main}, which states generally, the initially two-dimensional ground states in the linear Hamiltonian of honey lattice, or the eigenspace of the first eigenvalue in other words, bifurcate exactly into eight distinct curves of eigenstates in the nonlinear cases. It is remarkable result to exhibit such an bifurcation pattern, since it quite distinguish itself from many previous work reporting the nonlinear bifurcations. In those cases, the linear equations as the limit of nonlinear ones have one-dimensional ground state in general. When adding up the nonlinear term, the perturbed stationary states around the original ground state still remain one-dimensional. However, in this article, only the perturbation of some specific eigenfunctions in the degenerated two-dimensional eigenspace of the regressive linear Hamiltonian can result in the bifurcation of the nonlinear eigenstates. It looks like that the perturbation theory of degenerated eigenspace in physics monographs, but here we give a rigorous proof of the existence and the uniqueness of the eigenfunctions. It is still unknown if such bifurcation exists when in the strong potential or the norm of the testing function is large, though. In the derivation processes of the main results, the discrete symmetry of the system is found to have great influences in the determination of the allowed parameters that represent the states which eventually get involved in the bifurcation in the nonlinear cases. In honeycomb lattice, these symmetries are the reversion symmetry and the $C_3$ rotation symmetry. The results of this article may extend to any bifurcation of degenerated eigenspace, and we predict the discrete symmetry may play a similar role in the distribution of the allowed parameters.

This article is constructed as follows. In section \ref{sec-prop-linear}, we point out the basic knowledge of the linear Schr\"odinger equation in honeycomb lattice. Several properties of the primitive vectors, dual lattice, Dirac points and their eigenfunctions and eigenvalues are offered for reference. The result in the article \cite{Math-Honeycomb:Fefferman} is emphasized as the foundations for the following arguments. In section \ref{sec-intro-to-nonlinear}, we shortly restate the problem of the nonlinear Schr\"odinger equation. Some specific forms of the nonlinear terms and their applications are listed to show the potential of the researched model. In section \ref{sec-limitation-to-allowed-space}, we preliminarily showing the limitation of the potential candidates of the eigenfunctions. By the analytical methods, we orthogonally decompose the problem into the solvable system. The shallowing of the allowed parameter space implies the following discoveries of discrete bifurcation patterns. In section \ref{Sec-uni-and-radial-sepa}, we deduct two propositions concerning the uniqueness and the radial separability of the eigenfunctions. For the newly defined parameter space $\Sigma_q$, the propositions further investigate the topologies of the allowed subspace. In section \ref{sec-intro-to-nonlinear}, we obtain the main theorem \ref{theorem-main}, showing there are exactly eight bifurcation curves in general cases. In the deduction of the main theorem, major efforts are put into the construction of the eigenfunctions by a bootstrap method. To guarantee the procedure of bootstrapping, we also introduce the extended concept of pseudo eigenfunction as an intermediary step towards the true eigenfunctions. Six out of the eight modes are figured out by the careful consideration of the pseudo eigenfunctions.

\section{Properties of Linear Schr\"odinger Equation}\label{sec-prop-linear}
A brief review of the properties of the Schr\"odinger Equation with periodic potential is given in this section. By Bloch theorem and some symmetry arguments, we will see the degeneration of the eigenspace at K point in the momentum space and the distinctive shapes of the eigenfunctions due to the symmetry group of honeycomb lattice, which plays a crucial role in the deduction of eigenvalue problem of the nonlinear equation in the following sections.

Denote $\textbf{r}_1, \textbf{r}_2$ as the primitive vectors of the honeycomb lattice, i. e., the potential of the honeycomb lattice satisfy
\begin{eqnarray}
  V(\textbf{x}+\textbf{r}_1) &=& V(\textbf{x}), \\ \nonumber
  V(\textbf{x}+\textbf{r}_2) &=& V(\textbf{x}), \nonumber
\end{eqnarray}
where $\textbf{r} \in \mathbb{R}^2$ and
\begin{equation}
  \textbf{r}_1 = \left( \begin{array}{c}
                   \frac{\sqrt{3}}{2} \\
                   \frac{1}{2}
                 \end{array} \right), \, \textbf{r}_2 = \left( \begin{array}{c}
                   \frac{\sqrt{3}}{2} \\
                   -\frac{1}{2}
                 \end{array} \right) . \label{honeylattice}
\end{equation}
In the following passage, all the vectors would be of two dimension and written in bold form. According to the Bloch theorem, we know that the eigenspace of the periodic Hamiltonian can always be decomposed into smaller invariant space in respect with the momentum space. In other words, for any $\textbf{k} \in \mathbb{R}^2$, there is a solution $\psi(\textbf{r};
\textbf{k})$ satisfying the following equations:
\begin{equation}
  H \psi(\textbf{x};\textbf{k}) = E \psi(\textbf{x};\textbf{k}),
\end{equation}
\begin{equation}\label{psi-k}
\psi(\textbf{x}+\textbf{r};\textbf{k}) = e^{i\textbf{k} \cdot \textbf{r}} \psi(\textbf{x};\textbf{k}),
\end{equation}
where $\textbf{r}=N_1\textbf{r}_1+N_2\textbf{r}_2$ ($N_1$, $N_2\in \mathbb{Z}$) stands for any element of honeycomb lattice.
The Hamiltonian of a typical two-dimensional quantum system with periodic potential can be written as:
\begin{equation}\label{Hamiltonian_linear}
  H \equiv -\Delta + V(\textbf{x}),
\end{equation}
where $\Delta= \partial_{xx}+ \partial_{yy}$ is the Laplacian in $\mathbb{R}^2$. For each momentum $\textbf{k}$, we can also decompose the original eigenvalue problem into the following equivalent problems:
\begin{eqnarray}
  H(\textbf{k}) \phi(\textbf{x};\textbf{k}) &=& \mu(\textbf{k}) \phi(\textbf{x};\textbf{k}), \label{Hamiltonian-linear-k} \\
  H(\textbf{k}) &=& -(\nabla +i \textbf{k})^2+ V(\textbf{x}),\\
  \phi(\textbf{x}+\textbf{r};\textbf{k}) &=& \phi(\textbf{x};\textbf{k}).
\end{eqnarray}
Here $\phi(\textbf{k}) = \psi(\textbf{k})e^{-i \textbf{k} \cdot \textbf{x}}$. So the eigenfunctions are converted into periodic ones in terms of honeycomb translation symmetry. Moreover, given enough smoothness condition of the potential, all of these solutions consist of a complete basis in $L^2\left(\mathbb{R}^2/\{\textbf{r}\}\right)$, where $\mathbb{R}^2/\{\textbf{r}\}$ is the torus constructed by identifying the points whose difference is an element of honeycomb lattice. In the following we rewrite the lattice $\{\textbf{r}\}$ by $\Gamma$. For each periodic function $\phi(\textbf{x};\textbf{k})$, discrete Fourier transformation can be conducted to convert the function into the linear space of the dual lattice, the primitive basis of which are:
\begin{eqnarray}
  \textbf{k}_1 &=& \left(\begin{array}{c}
                           \frac{2\pi \sqrt{3}}{3}  \\
                           2\pi
                         \end{array} \right), \\
   \textbf{k}_2 &=& \left(\begin{array}{c}
                           \frac{2\pi \sqrt{3}}{3}  \\
                           -2\pi
                         \end{array} \right).
\end{eqnarray}
Therefore, all the allowed frequencies in the Fourier transformation of any $\phi(\textbf{x};\textbf{k})$ can be rewritten as $N_1\textbf{k}_1+N_2\textbf{k}_2$ ($N_1$, $N_2 \in \mathbb{Z}$).

For the honeycomb lattice, additional restrictions except for the translation symmetry are made in the periodic potential $V(\textbf{x})$ of Eq. (\ref{Hamiltonian_linear}). It is easy to enumerate all the sufficient and necessary conditions to build up a honeycomb lattice:

(1) V is periodic for any element of $\Gamma$.

(2) V has inversion symmetry, i.e. $V(-\textbf{x})=V(\textbf{x})$.

(3) V is invariant under clockwise rotation by $2\pi/3$, i.e. $V(R\textbf{x})=V(\textbf{x})$, where R is a 2 times 2 matrix:
\begin{equation}\label{rotation}
  R = \left( \begin{array}{cc}
               -\frac{1}{2} & -\frac{\sqrt{3}}{2} \\
               \frac{\sqrt{3}}{2} & -\frac{1}{2}
             \end{array}\right)
\end{equation}

The eigenvalue problem of Eq. (\ref{Hamiltonian-linear-k}) at certain specific parameters $\textbf{k}$ has some unusual properties, endowing materials like graphene with irreplaceable value in scientific researches and industry \cite{Graphene-review:RevModPhys.81.109}. One of the example is the corner of the Brillouin zone, or $\textbf{K}(\textbf{K'})$ point, the eigenvalue problem of which would be focused on later. Without loss of generality, we only show the properties of the reciprocal momentum space at K point, which is given by
\begin{equation}\label{Kpoint}
  \textbf{K} =\frac{4\pi}{3}\left( \begin{array}{c}
                        0 \\
                       1
                     \end{array} \right).
\end{equation}

One specificity of the K point is the rotational invariance of the function space $\{\psi(\textbf{x};\textbf{K})\}$. In detail, $\{\psi(R\textbf{x};\textbf{K})\}$ is also a eigenfunction of the Eq. (\ref{Hamiltonian_linear}) with parameter $\textbf{K}$. It implies that we can also define the rotation operator in the solution space $\{\phi(\textbf{x};\textbf{K})\}$ of Eq. (\ref{Hamiltonian-linear-k}).

To justify the above argument, take $f(\textbf{x})$ as a eigenfunction of Eq. (\ref{Hamiltonian_linear}) with parameter $\textbf{K}$. Due to the periodicity condition, f(\textbf{x}) has a representation of Fourier series:
\begin{equation}\label{Fourier}
  f(\textbf{x}) = \sum_{M_1, M_2 \in \mathbb{Z}} c(M_1,M_2) e^{i(\textbf{K}+ M_1\textbf{k}_1+M_2\textbf{k}_2) \cdot \textbf{x}}
\end{equation}
Rotating the coordinate by matrix $R$, we get:
\begin{eqnarray} \label{R-fourier}
    f(R\textbf{x})& = &\sum_{M_1, M_2 \in \mathbb{Z}} c(M_1,M_2) e^{i(\textbf{K}+ M_1\textbf{k}_1+M_2\textbf{k}_2) \cdot R\textbf{x}} \nonumber \\
   &=& \sum_{M_1, M_2 \in \mathbb{Z}} c(M_1,M_2) e^{i R^t(\textbf{K}+ M_1\textbf{k}_1+M_2\textbf{k}_2) \cdot \textbf{x}} \nonumber \\
   &=& \sum_{M_1, M_2 \in \mathbb{Z}} c(M_1,M_2) e^{i[\textbf{K}+ (-M_2)\textbf{k}_1+(M_1-M_2+1)\textbf{k}] \cdot \textbf{x}},
\end{eqnarray}
where $R^t$ is the transpose of the matrix $R$. It is an immediate result that the transformed function also has the form consistent with Eq. (\ref{psi-k}) as the original one. Furthermore, the commutativity of the Hamiltonian $H$ and the rotation operator $R$ implies that the transformed function is also an eigenfunction with the same eigenvalue.

It is also noticeable that we can now apply the rotation operator in the smaller subspace $\{\psi(\textbf{x};\textbf{K})\}$, or equivalently, $\{\phi(\textbf{x};\textbf{K})\}$, the eigenfunction space of the Hamiltonian with parameter $\textbf{K}$.

It is a proverbial fact by representation theory of finite group that the commutative group has only one-dimensional irreducible representations. As a result, every eigenfunction $\phi(\textbf{x};\textbf{K})$ of the effective Hamiltonian $H(\textbf{K})$ can be written as the sum of eigenfunctions which are also the eigenfunctions of the rotation operator. Since $R^3=I$, the eigenvalue of the rotation operator can only be $1, e^{2\pi i/3}, e^{-2\pi i/3}$. Denote $\omega=e^{2\pi i/3}$, and we further decompose the function space $\{\phi(\textbf{x};\textbf{K})\}$ into three smaller spaces:
\begin{eqnarray}\label{rot-space}
  L^2_{\textbf{K},1} &\equiv& \{\phi(\textbf{x};\textbf{K})|\widetilde{R} \phi(\textbf{x};\textbf{K}) = \phi(\textbf{x};\textbf{K}) \},  \\
  L^2_{\textbf{K},\omega} &\equiv& \{\phi(\textbf{x};\textbf{K})|\widetilde{R} \phi(\textbf{x};\textbf{K}) = \omega\phi(\textbf{x};\textbf{K}) \}, \\
  L^2_{\textbf{K},\overline{\omega}} &\equiv& \{\phi(\textbf{x};\textbf{K})|\widetilde{R} \phi(\textbf{x};\textbf{K}) = \overline{\omega}\phi(\textbf{x};\textbf{K}) \},
\end{eqnarray}
where $\overline{\omega}$ represents the conjugate of $\omega$ and $\widetilde{R}$ is a transformed rotation operator: $\widetilde{R}f(\textbf{x})= e^{-i\textbf{K}\cdot \textbf{x}}\left[e^{i\textbf{K} \cdot (\cdot)}f(\cdot)\right](R\textbf{x})$. An additional noteworthy remark is that if there exists an eigenfunction $f(\textbf{x}) \in L^2_{\textbf{K},\omega}$ of the Hamiltonian $H(\textbf{K})$, then $\overline{f(-\textbf{x})}$ lying in $L^2_{\textbf{K},\overline{\omega}}$ is also an eigenfunction of the same eigenvalue. Moreover, for $g(\textbf{x}) \in L^2_{\textbf{K},\omega}$, $h(\textbf{x}) \in L^2_{\textbf{K},\overline{\omega}}$ and $M(\textbf{x})$ a real function of honeycomb lattice symmetry, we have the following equality:
\begin{equation}\label{vanishing-honeycomb}
  I \equiv \int_{\mathbb{R}^2/\Gamma}M(\textbf{x}) \overline{g(\textbf{x})} h(\textbf{x}) d^2\textbf{x} = \int_{\mathbb{R}^2/\Gamma}M(\textbf{x}) g(\textbf{x}) \overline{h(\textbf{x})} d^2\textbf{x} = 0.
\end{equation}
It can be derived by applying the operator $\widetilde{R}$ to the integrand and by the invariance of the measure under the operator. This operation gives $(\overline{\omega}^2-1)I=0$ ($(\omega^2-1)I=0$) and then Eq. (\ref{vanishing-honeycomb}) follows. The rotational properties of the function space are quite useful and some deductions of similar identities are repeated constantly in the following passage.

There are plenties of researches focusing on the eigenvalue problem of the linear system $H(\textbf{K})$. To name a few, Fefferman and Weinstein have proved the following theorem, showing the degeneration of the eigenspace of the first eigenvalue for some specific weak honeycomb-like potential.
\begin{theorem}[Fefferman and Weinstein\cite{Math-Honeycomb:Fefferman}]
For $V(\textbf{x})$ a honeycomb lattice, suppose one of the Fourier coefficient of $V(\textbf{x})$ is nonzero:
\begin{equation}\label{1-1fourier}
  V_{1,1}=\int_{\mathbb{R}^2/\Gamma} e^{-(\textbf{k}_1+\textbf{k}_2)\cdot \textbf{x}} V(\textbf{x}) d^2\textbf{x} \neq 0.
\end{equation}
Then for sufficiently small $\epsilon$ such that $\epsilon V_{1,1}>0$, the eigenspace of the first eigenvalue of $H_{\textbf{K}}= -(\nabla + i\textbf{K})^2+\epsilon V(\textbf{x})$ is two-dimensional, with $\phi_0^a \in L^2_{\textbf{K},\omega}$ and $\phi_0^b \in L^2_{\textbf{K},\overline{\omega}}$ as the two linearly independent eigenfunctions. For proper choices, $\phi_0^a(\textbf{x})=\overline{\phi_0^b(-\textbf{x})}$.
\end{theorem}

The lowest degenerate two-dimensional linear space, the bifurcation of which with respect to the parameter $\textbf{k}$ in the neighborhood of $\textbf{K}$ is known as the famous Dirac cone, lies exactly in $L^2_{\textbf{K},\omega} \oplus L^2_{\textbf{K},\overline{\omega}}$. In condensed matter physics, this conical dispersion of the electronic spectral implies a type of dynamics dominated by the equation for the massless Dirac fermion \cite{Graphene-review:RevModPhys.81.109}. In the following sections, we would take advantage of these symmetrical properties to construct a bootstrap method in the analysis of eigenvalue problem of the nonlinear Schr\"odinger equation in the honeycomb lattice.

\section{Nonlinear Schr\"odinger Equation of Honeycomb Lattice}\label{sec-intro-to-nonlinear}
In this section, we would focus on the nonlinear schrodinger equation (NLS) in a honeycomb lattice, where the nonlinear Hamiltonian is:
\begin{equation}\label{NLS-original}
  H = -\Delta + V_{NL}(\textbf{x},|\psi(\textbf{x})|^2),
\end{equation}
where $V_{NL}$ is a $C^\infty$ function of $\psi$ and $\textbf{x}$ and $\psi(\textbf{x})$ is the function which the Hamiltonian is applied to. There are numerous examples of physical models that can be described by Eq. (\ref{NLS-original}). Here we present two kinds of the nonlinear effects that attract great interests in the scientific community:

(1) Kerr terms \cite{honey-NLS:EXP-Omri,Math-Nonlinear-system-Yang}, or $V_{NL} = V_L+ K|\psi|^2$, where $V_L$ is a honeycomb lattice potential independent of $\psi$. It is one of the most simplistic nonlinear term in the nonlinear dynamics. It is also called Gross-Pitaevskii equation (GPE) in the researches of Bose-Einstein condensates. In the language of quantum field theory, it is also regarded as the variation equation of the complex scalar Hamiltonian with $\phi^4$ term. It also consist of an integrable system if the linear term $V_{L}$ is trivial. This model is well suitable for the case when the strength of the wave is not strong, i. e. the $L^2$ norm is not large, say. The integrability also helps to give the analytical soliton solutions, for which why this model is so famous.

(2) Saturable nonlinear Schr\"odinger equation \cite{Honey-NLS:LinTai-Chia-Saturable-energy,Math-Nonlinear-system-Yang}, with the potential as $V_{NL}= K/(1+V_{L}+|\psi|^2)$. In this case, the nonlinear is globally bounded for $V_{L}>0$ for any points in $\mathbb{R}^2$. It is useful for strong external fields or testing fields, usually discussed in the context of the propagation of electromagnetic field in quantum optics. We will see the differences of the saturable nonlinear term with the Kerr terms in the determination of the eigenfunctions among the pseudo ones.

Although the Bloch theorem fails to decompose the nonlinear system into subproblems with specified momentum $\textbf{k}$, we can also consider the eigenvalue problem within the function space of certain periodicity condition as an approximation for the limit of wavepacket with short width in the momentum space. Specifically, we can define the similar momentum-dependent Hamiltonian:
\begin{equation}\label{NLS-k}
  H(\textbf{k})= -(\nabla +i \textbf{k})^2 + V_{NL}.
\end{equation}
For $\textbf{k}$ lies in the bulk of the Brillouin zone, plenties of researches showed usually all the eigenvalues of the linear Hamiltonian $H(\textbf{k})$ in Eq. (\ref{Hamiltonian-linear-k}) are of multiplicity one \cite{Graphene-review:RevModPhys.81.109}. For the corresponding models of Eq. (\ref{NLS-k}), it is actually equivalent to a double-well model in the tight-binding approximation \cite{NLS:Smerzi-original-PhysRevLett}, with the same bifurcation pattern of the ground state energy, the physical jargon of the first eigenvalue, as the varying of the norm of the eigenfunctions as reported. The equivalence of these two models will be discussed in the forthcoming article of the author. Apart from the bulk cases, things are changed in the Dirac points due to the degeneration of the eigenspace. In the following passage, we are going to show the quite special bifurcation of the eigenfunctions.

\section{The restriction of parameter space in the eigenvalue problem}\label{sec-limitation-to-allowed-space}
For the sake of construction of eigenfunctions, we consider the approximate linear differential equations, of which the Hamiltonian is:
\begin{equation}\label{LS-series}
  H_t= -\Delta + V_{NL}(\textbf{x},|\psi_t(\textbf{x})|^2),
\end{equation}
where $\psi_t$ acted as a test function would be given in different situations. Since now the function inside the nonlinear term is fixed, the Hamiltonian retains its classical definition with linearity. Subsequently, we also define the corresponding Hamiltonian with momentum $\textbf{k}$:
\begin{eqnarray}
H_t(\textbf{k}) &=& -(\nabla +i \textbf{k})^2 + \widetilde{V}_{NL} (\textbf{x},| \phi_t(\textbf{x}) |^2 ), \nonumber \\
   &\equiv& -\Delta_{\textbf{k}} + V_{NL}, \label{LS-series-k}
\end{eqnarray}
if $\psi_t$ is also a Bloch function of the honeycomb lattice.

Recall that the ground states consist of a two-dimension linear subspace of the whole Hilbert space $L^2(\mathbb{R}^2/\Gamma)$. Denote the basis of the subspace as $\psi_0^a$ and $\psi_0^b$, which are of norm one lying in the subspace $L^2_{\textbf{K},\omega}$ and $L^2_{\textbf{K},\overline{\omega}}$ defined in section \ref{sec-prop-linear} respectively. Remark that $\psi_0^a(\textbf{x})=\overline{\psi_0^b(-\textbf{x})}$ as mentioned in section \ref{sec-prop-linear}. To begin the bootstrap argument, we set $\psi_t=a\psi_0^a+b\psi_0^b$. Below we show that by a perturbation method, the eigenfunction of Eq. (\ref{LS-series-k}) for $\psi_t$ can only exist in a restricted parameter area for $(a,b)$. Indeed, we have an even more general proposition for many $\psi_t$:

\begin{proposition}\label{Prop-1LS}
For sufficiently small $\epsilon>0$, suppose $V_{NL}(\textbf{x},|\phi_t(\textbf{x})|^2)$ in Eq. (\ref{LS-series-k}) is a $C^\infty(\textbf{x},|\phi_t|^2)$ function expanded as $V_{NL}=V_{L}(\textbf{x})+v(\textbf{x},|\phi_t(\textbf{x})|^2)$, where $v = K(\textbf{x})|\phi(\textbf{x})|^2+ O(|\phi(\textbf{x})|^4)$  and $K(\textbf{x})$ is a nonzero function of honeycomb symmetry. Additionally, $\phi_t \in H^2(\mathbb{R}^2/\Gamma)$ is chosen to satisfy
\begin{equation}\label{P1Q-LS-NLterm-condition1}
  1- \frac{\left|\langle \phi_t, \phi_0^a \rangle\right|^2+\left|\langle \phi_t, \phi_0^b \rangle\right|^2}{\langle \phi_t, \phi_t \rangle} < \epsilon^4
\end{equation}
and $\|\phi_t(\textbf{x})\|_{L^2(\mathbb{R}^2/\Gamma)}<\epsilon$.$ K(\textbf{x})$ satisfies
\begin{eqnarray}
  \nonumber && \int_{\mathbb{R}^2/\Gamma} K(\textbf{x}) \left[|\phi_0^a(\textbf{x})|^4-2|\phi_0^b(\textbf{x})|^2 |\phi_0^a(\textbf{x})|^2 \right] d^2\textbf{x}  \\
   \label{P1Q-LS-NLterm-condition2}
 &=& \int_{\mathbb{R}^2/\Gamma} K(\textbf{x}) \left[|\phi_0^b(\textbf{x})|^4-2|\phi_0^b(\textbf{x})|^2 |\phi_0^a(\textbf{x})|^2 \right] d^2\textbf{x} \neq 0.
\end{eqnarray}
Then there exists $\delta>0$ so that the necessary condition that there exists an eigenvalue function $\phi(\textbf{x})$ in $H^2(\mathbb{R}^2/\Gamma)$ of Eq. (\ref{LS-series-k}), such that $\langle \phi(\textbf{x})-\phi_t(\textbf{x}), \phi_0^a(\textbf{x}) \rangle=\langle \phi(\textbf{x})-\phi_t(\textbf{x}), \phi_0^b(\textbf{x}) \rangle=0$ and $(H_t(\textbf{k})-E_0)\|\phi(\textbf{x})\|_{L^2(\mathbb{R}^2/\Gamma)}<\epsilon \|\phi(\textbf{x})\|_{L^2(\mathbb{R}^2/\Gamma)}$, is that either
\begin{equation}\label{P1Q-condition-1}
  \frac{\langle\phi_0^a,\phi\rangle}{\langle\phi,\phi\rangle}<\delta  \quad or \quad \frac{\langle\phi_0^b,\phi\rangle}{\langle\phi,\phi\rangle}<\delta,
\end{equation}
or
\begin{equation}\label{P1Q-condition-2}
  \left| |\langle \phi_0^a,\phi\rangle|^2 - |\langle \phi_0^b,\phi \rangle|^2 \right| < \delta.
\end{equation}
Here $\langle \cdot, \cdot \rangle$ means the inner product in the space $L^2(\mathbb{R}^2/\Gamma)$ and $E_0$ is the first eigenvalue of the function-independent $H(\textbf{K})$, the linear Hamiltonian Eq. (\ref{LS-series-k}) with $\phi_t=0$.
\end{proposition}

\begin{proof}
Suppose the existence of the eigenfunction of Eq. (\ref{LS-series-k}). To find an eigenfunction of small norm, we can expand the supposed eigenfunction $\phi(\textbf{x})$ as
\begin{equation}\label{eigenf-expansion}
  \phi = \epsilon( a\phi_0^a+b\phi_0^b)+ \widetilde{\phi}_1,
\end{equation}
where $|a|^2+|b|^2=1$ and $\widetilde{\phi}_1$ is orthogonal to $\phi_0^a$ and $\phi_0^b$ in $L^2$ space. Suppose the eigenvalue of $\phi$ is $E_0+E_1$ and substitute Eq. (\ref{eigenf-expansion}) into Eq. (\ref{LS-series-k}). We have
\begin{eqnarray}
  \nonumber && \left[-\Delta_{\textbf{K}} \epsilon(a\phi_0^a+b\phi_0^b) + V_{L} \epsilon(a\phi_0^a+b\phi_0^b) \right] - \Delta_{\textbf{K}} \widetilde{\phi}_1 + V_{L} \widetilde{\phi}_1  +v(|\phi_t|^2) \widetilde{\phi}_1   \\
     &=& v(|\phi_t|^2) \epsilon(a\phi_0^a+b\phi_0^b) + (E_0+E_1) \widetilde{\phi}_1 + E_1 \epsilon(a\phi_0^a+b\phi_0^b). \label{P1-original}
\end{eqnarray}
By the definition of $\phi_0^a$ and $\phi_0^b$, the first term in the left hand side vanishes. Therefore, we derive a linear, nonhomogeneous elliptical function for $\widetilde{\phi}_1$. Define $L=-\Delta_{\textbf{K}} + V_{L} - E_0$. To solve this partially differential equation, we decompose both sides of the equation into two orthogonal spaces by two operators $M_{\|}$ and $M_{\perp}$, which project $L^2(\mathbb{R}^2/\Gamma)$ into $\{\phi_0^a\} \oplus \{\phi_0^b\}$ and its orthogonal complementary space, respectively. In other words,
\begin{equation}
  M_{\perp}\phi_0^a=M_{\perp}\phi_0^b=0, \quad M_{\perp}\widetilde{\phi}_1 = \widetilde{\phi}_1
\end{equation}
and $M_{\|}=I-M_{\perp}$. Then we transform Eq. (\ref{P1-original}) into two equations:
\begin{eqnarray}
\label{P1-2Eqs-1} (L-E_1) \widetilde{\phi}_1 + M_{\perp} [v(|\phi_t|^2)\widetilde{\phi}_1] &=& -M_{\perp} [v(|\phi_t|^2)\epsilon(a\phi_0^a+b\phi_0^b)],  \\
  \label{P1-2Eqs-2} M_{\|} [v(|\phi_t|^2)\widetilde{\phi}_1] &=& E_1\epsilon(a\phi_0^a+b\phi_0^b).
\end{eqnarray}
Here we use $\langle \phi_0^a ,L\widetilde{\phi}_1\rangle=\langle \phi_0^b ,L\widetilde{\phi}_1\rangle=0$. Since the elliptic operator has a discrete spectrum, it is readily obtained that $L$ is a reversible operator in the space $M_{\perp}L_2(\mathbb{R}^2/\Gamma)$. Therefore, given $\epsilon$ and $E_1$ are sufficiently small, we have a unique solution of Eq. (\ref{P1-2Eqs-1}):
\begin{equation}\label{P1-solution}
  \widetilde{\phi}_1= -(1+ L^{-1}M_{\perp}[v(|\phi_t|^2)\cdot (\cdot)]-E_1L^{-1})^{-1} L^{-1}M_{\perp} [v(|\phi_t|^2)\epsilon(a\phi_0^a+b\phi_0^b)],
\end{equation}
where $v(|\phi_t|^2) \cdot (\cdot)$ represent the operator of multiplying a $C^\infty$ function $v(\phi_t)$ if $\phi_t \in C^\infty(\mathbb{R}^2/\Gamma)$. If $E_1$ is supposed to be small enough, the invertibility of the operator $1- L^{-1}M_{\perp}[v(|\phi_t|^2)\cdot (\cdot)+E_1]$ results from the small norm of $v(|\phi_t|^2)$ and $E_1$. Indeed, by elliptic regularity, this operator should be a reversible mapping in $M_{\perp}H^s(\mathbb{R}^2/\Gamma)$ for any s. See lemma \ref{lemma-HSregularity} in the appendix. So the $\widetilde{\phi}_1$ is also a $C^\infty$ function. Note that $\|\widetilde{\phi}_1\|_{L^2} \sim O(\epsilon^3)$ from Eq. (\ref{P1-solution}).

Now it is time to examine the consistency of the $\widetilde{\phi}_1$ given by Eq. (\ref{P1-2Eqs-1}) with Eq. (\ref{P1-2Eqs-2}). Substitute the expression of $\widetilde{\phi}_1$ into Eq. (\ref{P1-2Eqs-2}) and calculate the inner product in each side with $\phi_0^a$ and $\phi_0^b$ and then we have the following two consistency conditions:
\begin{eqnarray}
\label{P1-rev-condition-1}  \int_{\mathbb{R}^2/\Gamma} v(|\phi_t|^2)\left[\epsilon a \left|\phi_0^a\right|^2 + \epsilon b \overline{\phi_0^a}\phi_0^b+ \overline{\phi_0^a}\widetilde{\phi}_1 \right]d^2\textbf{x} &=& \epsilon E_1 a, \\
 \label{P1-rev-condition-2} \int_{\mathbb{R}^2/\Gamma} v(|\phi_t|^2)\left[\epsilon a \overline{\phi_0^b}\phi_0^a + \epsilon b \left|\phi_0^b\right|^2 + \overline{\phi_0^b}\widetilde{\phi}_1 \right]d^2\textbf{x} &=& \epsilon E_1 b.
\end{eqnarray}
Eliminate $E_1$ by the linear combinations $b\cdot$Eq. (\ref{P1-rev-condition-1})$-a\cdot$Eq. (\ref{P1-rev-condition-2}) and the consistency condition is transferred to be
\begin{equation}\label{P1-consistency}
  ab(E_a-E_b) = a^2{\overline{E}}_{int}-b^2E_{int} +\frac{1}{\epsilon}\int_{\mathbb{R}^2/\Gamma} v(|\phi_t|^2)\widetilde{\phi}_1[b\overline{{\phi}_0^a}-a\overline{\phi_0^b}] d^2\textbf{x},
\end{equation}
where
\begin{eqnarray}
  E_a &=& \int_{\mathbb{R}^2/\Gamma} v(|\phi_t|^2) \left| \phi_0^a \right|^2 d^2\textbf{x}, \\
  E_b &=& \int_{\mathbb{R}^2/\Gamma} v(|\phi_t|^2) \left| \phi_0^b \right|^2 d^2\textbf{x}, \\
  E_{int} &=& \int_{\mathbb{R}^2} v(|\phi_t|^2) \overline{\phi_0^a} \phi_0^b d^2\textbf{x}.
\end{eqnarray}
According to the condition satisfied by $\widetilde{\phi}_1$ (Eq. (\ref{P1-solution})) and $v(|\phi_t|^2)$ (Eq. (\ref{P1Q-LS-NLterm-condition1})), it is readily known from the expansion about $\epsilon$ that the last term in the right hand side of Eq.  (\ref{P1-consistency}) is of order $O(\epsilon^4)$ and
\begin{eqnarray}
  E_a &=& \epsilon^2 \left[|a|^2 I_a + |b|^2 I_{int}\right] +O(\epsilon^4),\\
  E_b &=& \epsilon^2 \left[|a|^2 I_{int} + |b|^2 I_b\right] + O(\epsilon^4),\\
  E_{int} &=& \epsilon^2 \left[a\overline{b} I_{int}\right] + O(\epsilon^4),
\end{eqnarray}
where $I_a = \int K|\phi_0^a|^4$, $I_b = \int K|\phi_0^b|^4 $ and $I_{int}=\int K|\phi_0^a|^2|\phi_0^b|^2$. Here Eq. (\ref{vanishing-honeycomb}) is used. Since $I_a=I_b$ derived by the properties of these two function aforementioned, denote $I_{one} \equiv I_a = I_b$. Therefore, Eq. (\ref{P1-consistency}) is rewritten as
\begin{equation}\label{P1-consistency-2}
  ab(|b|^2-|a|^2)[I_{one}-2I_{int}]+J(\epsilon)=0,
\end{equation}
where $J \sim O(\epsilon^2)$. Let $|P| < C\epsilon^2$. Suppose the converse of conditions Eqs. ($\ref{P1Q-condition-1}$) and ($\ref{P1Q-condition-2}$) in Proposition \ref{Prop-1LS}. If $\delta=|C\epsilon/(I_{one}-2I_{int})|^{1/3}$, the left hand side of Eq. (\ref{P1-consistency-2}) must be greater than 0, contradicting the existence of the eigenfunction of (\ref{LS-series-k}).
\end{proof}

\section{Uniqueness and Radial separability of the allowed parameter space}\label{Sec-uni-and-radial-sepa}
We have discussed the necessary conditions for ones to find the eigenfunctions of the nonlinear system. Only in a highly restricted parameter space for $a$ and $b$ mentioned above can the existence of the eigenfunction be proved. It gives a strong hint that the original two-dimensional linear space of the stationary wavefunction space spontaneously decays into limited cases in the parameter space. Now we further investigate the properties of the parameter space.

In the following passage, we define the full parameter space $\Sigma = \{(a,b)| |a|^2 +|b|^2 =1 \}$. Here $a$ and $b$ share the same meaning as defined in section \ref{sec-limitation-to-allowed-space}. Apparently, the eigenfunction is physically unchanged if multiplied by an constant complex number of norm $1$. This equivalence can be even expanded if consider the linear system, since now the eigenfunction forms a linear space and is even unchanged after multiplication by a constant without the restriction of norm 1, although it is not the case here. For this reason, we redefine a new quotient space of $\Sigma$ as $\Sigma_q = \Sigma / T$, where $T$ is a set of binary equivalence relations $\{(a,b)\sim (c,d)| ad-bc=0\}$. If the topology of $\Sigma_q$ is inherited from the natural topology of metric space $\mathbb{C}^2$, then it is readily seen the homeomorphism of $\Sigma_q$ to $S^2$ in $\mathbb{R}^2$.

For the nonlinear equation, we subsequently define the allowed parameter space $\Sigma_q^a(\epsilon)$, which is the set of all the allowed pairs $(a,b,\epsilon)$ for which there at least exists one eigenfunction $\phi \in H^2$ of the nonlinear Hamiltonian with the eigenvalue $E$ such that $\|\phi\|_{L^2}\le \epsilon$ and $|E-E_0|\le \epsilon$, where we set $\epsilon<\epsilon_0$ such that Eq. (\ref{P1-solution}) is well-defined. Given all the preparation, we firstly consider the Lipschitz continuity of $\phi$ in the allowed parameter space:
\begin{lemma}\label{Lemma1-Lpcont}
For sufficiently small $\epsilon>0$, suppose $(a^{(1)},b^{(1)})$, $(a^{(2)},b^{(2)})$ are both pairs in $\Sigma_q^a(\epsilon)$. Then for any pairs $(a^{(1)},b^{(1)},E^{(1)},\phi^{(1)})$ and $(a^{(2)},b^{(2)},E^{(2)},\phi^{(2)})$ satisfying the condition of Proposition \ref{Prop-1LS}, we have the following relation:
\begin{equation}\label{L1-Lipschitz}
  \|M_{\perp}(\phi^{(1)}-\phi^{(2)})\|_{H^s(\mathbb{R}^2/\Gamma)} \le C\left(|a^{(1)}-a^{(2)}|+|b^{(1)}-b^{(2)}|\right)\epsilon^3
\end{equation}
for any $s \ge 2$, where $C$ is only dependant of $V_{NL}$ and $s$.
\end{lemma}
\begin{proof}
Substitute the two pairs $(a^{(i)},b^{(i)},E^{(i)},\phi^{(i)})$ $(i=1,2)$ into Eq. (\ref{P1-2Eqs-1}) and calculate the difference
\begin{eqnarray}
  \nonumber & & L\left(\widetilde{\phi}^{(1)}_1-\widetilde{\phi}^{(2)}_1\right)-\left(E_1^{(1)}\widetilde{\phi}^{(1)}_1-E_{(2)}\widetilde{\phi}^{(2)}_1\right)+M_{\perp}\left[v(|\phi^1|^2)\left(\widetilde{\phi}_1^{(1)}-\widetilde{\phi}_1^{(2)}\right)\right]\\
  \nonumber & & +M_{\perp}\left[\widetilde{\phi}_1^{(2)}\left(v\left(|\phi^{(1)}|^2\right)-v\left(|\phi^{(2)}|^2\right)\right)\right] \\
 \nonumber &=&  -\epsilon M_\perp \left[v\left(|\phi^{(1)}|^2\right)\left(a^{(1)}-a^{(2)}\right)\phi_0^a\right]-\epsilon M_\perp \left[v\left(|\phi^{(1)}|^2\right)\left(b^{(1)}-b^{(2)}\right)\phi_0^b\right]  \\
 & &-\epsilon M_\perp \left[\left(v\left(|\phi^{(1)}|^2\right)-v\left(|\phi^{(2)}|^2\right)\right) \left(a^{(2)}\phi_0^a +b^{(2)} \phi_0^b\right)\right].
\end{eqnarray}
For $v(x) \in C^{\infty}(\mathbb{R})$, it is readily known that $D(x,y)=\frac{v(x)-v(y)}{x-y}$ is a $C^{\infty}(\mathbb{R}^2)$ function. Therefore, we further transfer the equation as
\begin{eqnarray}
  \nonumber  L\left(\widetilde{\phi}^{(1)}_1-\widetilde{\phi}^{(2)}_1\right)&=&E_1^{(1)}\left(\widetilde{\phi}^{(1)}_1-\widetilde{\phi}^{(2)}_1\right)-M_{\perp}\left[v\left(|\phi^{(1)}|^2\right)\left(\widetilde{\phi}_1^{(1)}-\widetilde{\phi}_1^{(2)}\right)\right] \\
   \nonumber &-& M_\perp\left[ D\left(|\phi^{(1)}|^2,|\phi^{(2)}|^2\right)\left(|\phi^{(1)}|^2-|\phi^{(2)}|^2\right)\phi^{(2)}\right] \\
   \nonumber &-& \epsilon M_\perp\left[\left(\left(a^{(1)}-a^{(2)}\right)\phi_0^a+\left(b^{(1)}-b^{(2)}\right)\phi_0^b\right) v(|\phi^{(1)}|^2)\right] \\
  & +& \left(E^{(2)}_1-E^{(1)}_1\right)\widetilde{\phi}_1^{(2)}.
\end{eqnarray}
The properties to be unravelled are the difference $E^{(2)}-E^{(1)}$. Here we invoke the Eqs. (\ref{P1-rev-condition-1}) and (\ref{P1-rev-condition-2}). Note that we only need to consider the case when $a^{(1)}-a^{(2)}$ and $b^{(1)}-b^{(2)}$ are sufficiently small. For both the differences are greater than a certain number, say $1/2\sqrt{2}$, we only need to prove $\widetilde{\phi}_1^{(i)} \sim O(\epsilon^3)$, which is already proved in Proposition \ref{Prop-1LS}. Then suppose $a^{(1)}>1/\sqrt{2}$ and $a^{(1)}-a^{(2)}< 1/(2\sqrt{2})$ without loss of generality. We just need to pick Eq. (\ref{P1-rev-condition-1}) and get:
\begin{eqnarray}
  &&\nonumber E^{(1)}_1-E^{(2)}_1 = \int_{\mathbb{R}^2/\Gamma} \left\{v(|\phi^{(1)}|^2)\left[\left(\frac{b^{(1)}}{a^{(1)}}-\frac{b^{(2)}}{a^{(2)}}\right) \phi_0^b\overline{\phi_0^a} + \frac{\left(\widetilde{\phi}_1^{(1)}-\widetilde{\phi}_1^{(2)}\right)\overline{\phi_0^a}}{\epsilon a^{(1)}} -\frac{\left(a^{(1)}-a^{(2)}\right)\widetilde{\phi}_1^{(2)}\overline{\phi_0^a}}{\epsilon a^{(1)}a^{(2)}} \right] \right. \\
  &+& \left. D\left(|\phi^{(1)}|^2,|\phi^{(2)}|^2\right)\left(|\phi^{(1)}|^2-|\phi^{(1)}|^2\right)\left[ \epsilon  |\phi_0^a|^2 +\epsilon \frac{b^{(2)}}{a^{(2)}} \overline{\phi_0^a} \phi_0^b +\frac{1}{a^{(2)}} \overline{\phi_0^a}\widetilde{\phi}_1^{(2)} \right] \right\} d^2\textbf{x}. \label{L1-eigenvalue-comp}
\end{eqnarray}
To disentangle the absolute value sign, note that for any $A,B,C,D$, $|A+B|-|C+D| \le |A+B-C-D| \le |A-C|+ |B-D|$. Let $A=\widetilde{\phi}_1^{(1)}$, $B=a^{(1)}\phi_0^a+b^{(1)}\phi_0^b$, $C=\widetilde{\phi}_1^{(2)}$ and $D=a^{(2)}\phi_0^a+b^{(2)}\phi_0^b$. In sum, we can estimate $\widetilde{\phi}_1^{(1)}-\widetilde{\phi}_1^{(2)}$ as
\begin{equation}
\|\widetilde{\phi}_1^{(1)} -\widetilde{\phi}_1^{(2)} \|_{H^s} \le C_1\epsilon^2 \|\widetilde{\phi}_1^{(1)} -\widetilde{\phi}_1^{(2)} \|_{H^s} + C_2\epsilon^3\left(|a^{(1)}-a_{(2)}|+|b_{(1)}-b_{(2)}|\right).
\end{equation}
Here we take use of the finite norms of $L^{-1}$ and $M_\perp$, the $C^{\infty}$ properties of $v(x)$ and several estimations: $v(|\phi^{(i)}|^2) \sim O (\epsilon^2)$, $D(|\phi_1^{(1)}|^2,|\phi_1^{(2)}|^2)\sim O(\epsilon)$ and $\widetilde{\phi}_1^{(i)} \sim O(\epsilon^3)$. For the $H^s$ regularity, we use the arguments in Lemma \ref{lemma-HSregularity}. Then the proof is complete.
\end{proof}

For $a^{(1)}=a^{(2)}$, $b^{(1)}=b^{(2)}$, we readily obtain $\widetilde{\phi}_1^{(1)}=\widetilde{\phi}_1^{(2)}$, which means the uniqueness of the solution:
\begin{proposition}[Uniqueness]\label{Prop-uniqueness}
For $(a,b)$ an allowed pair in $\Sigma_q^a (\epsilon)$, there is only one eigenfunction $\phi$ of Eq. (\ref{NLS-k}), or in an equivalent meaning, an eigenfunction of Eq. (\ref{LS-series-k}) with $\phi_t=\phi$ that satisfy the conditions in Proposition \ref{Prop-1LS}.
\end{proposition}

It endows each allowed parameter pair a bijection towards an eigenfunction. A more crucial fact is we can actually prove the following proposition of the radial separability of the allowed parameter space in small $\epsilon$. Based on the restriction of allowed parameter space given in Proposition \ref{Prop-1LS}, below the Proposition \ref{prop-rad-separability} excludes a wide internal around a given allowed pair of the whole parameter space in the radial direction. It indeed implies the allowed parameter space can only be at most one-dimensional, showing a giant progress toward the main theorem.

\begin{proposition}[Radial Separability]\label{prop-rad-separability}
For the pair $\left(a^{(1)},b^{(1)}\right)=(\cos(\theta) e^{i\alpha}, \sin(\theta) e^{i\beta})$ in $\Sigma_q^a(\epsilon)$. Suppose there exists an eigenfunction $\phi$ of Eq. (\ref{NLS-k}), i.e. an eigenfunction of Eq. (\ref{LS-series-k}) with $\phi_t=\phi$, that satisfies the condition given in Proposition \ref{Prop-1LS}. Then the following statements of the radial separability hold:

(1) If there exists $\delta>0$ such that $|a^{(1)}| \ge \delta$ and $|b^{(1)}| \ge \delta$, then for sufficiently small $\epsilon$, there is not another distinct allowed pair $\left(a^{(2)},b^{(2)}\right)=(\cos(\theta') e^{i\alpha}, \sin(\theta') e^{i\beta})$ for which there exists an eigenfunction of Eq. (\ref{NLS-k}) that also satisfies the same condition while $|a^{(2)}| \ge \delta$ and $|b^{(2)}| \ge \delta$.

(2) If $a^{(1)}=0$ or $b^{(1)}=0$, then there exist $\delta>0$ such that for sufficiently small $\epsilon>0$, there is not another distinct allowed pair $(a^{(2)},b^{(2)})$ for which there exists an eigenfunction of Eq. (\ref{NLS-k}) that also satisfies the same condition while $|a^{(2)}| \le \delta$ or $|b^{(2)}| \le \delta$.
\end{proposition}

\begin{proof}
We start the proof from Eqs. (\ref{P1-rev-condition-1}) and (\ref{P1-rev-condition-2}). Substitute $(a^{(i)},b^{(i)},E^{(i)},\phi^{(i)})$ $(i=1,2)$ into the equations and calculate the differences we have:
\begin{eqnarray}
 \nonumber && E_1^{(1)}\left(a^{(1)}-a^{(2)}\right)+\left(E_1^{(1)}-E_1^{(2)}\right) \\
  \nonumber &=& \int_{\mathbb{R}^2/\Gamma} \left\{v\left(|\phi^{(1)}|^2\right)\left[\left(a^{(1)}-a^{(2)}\right)|\phi_0^a|^2+\left(b^{(1)}-b^{(2)}\right) \overline{\phi_0^a}\phi_0^b + \frac{\overline{\phi_0^a} \left(\widetilde{\phi}_1^{(1)}-\widetilde{\phi}_1^{(2)}\right)}{\epsilon} \right] \right.  \\
    &+& \left.\left[v\left(|\phi_1^{(1)}|^2\right) - v\left(|\phi_1^{(2)}|^2\right)\right]\left(a^{(2)}|\phi_0^a|^2+b^{(2)}\overline{\phi_0^a}\phi_0^b +\frac{\overline{\phi_0^a} \widetilde{\phi}_1^{(2)}}{\epsilon}\right) \right\} d^2\textbf{x},\\
 \nonumber && E_1^{(1)}\left(b^{(1)}-b^{(2)}\right)+\left(E_1^{(1)}-E_1^{(2)}\right) \\
 \nonumber &=& \int_{\mathbb{R}^2/\Gamma} \left\{v\left(\phi^{(1)}\right)\left[\left(b^{(1)}-b^{(2)}\right)|\phi_0^b|^2+\left(a^{(1)}-a^{(2)}\right) \overline{\phi_0^b}\phi_0^a + \frac{\overline{\phi_0^b} \left(\widetilde{\phi}_1^{(1)}-\widetilde{\phi}_1^{(2)}\right)}{\epsilon} \right] \right.  \\
    &+& \left.\left[v\left(\phi_1^{(1)}\right) - v\left(\phi_1^{(2)}\right)\right]\left(b^{(2)}|\phi_0^b|^2+a^{(2)}\overline{\phi_0^b}\phi_0^a +\frac{\overline{\phi_0^b} \widetilde{\phi}_1^{(2)}}{\epsilon} \right) \right\} d^2\textbf{x}.
\end{eqnarray}
By $v(|\phi|^2)=K(\textbf{x})|\phi|^2+O(|\phi|^4)$ and Lemma \ref{Lemma1-Lpcont}, we can figure out the terms of order $\epsilon^2$:
\begin{eqnarray}
 \nonumber && E_1^{(1)} \left(a^{(1)}-a^{(2)}\right)+\left(E_1^{(1)}-E_1^{(2)}\right)a^{(2)} \\
 \nonumber &=& \int_{\mathbb{R}^2/\Gamma} \left\{ \epsilon^2\left[\left(a^{(1)}-a^{(2)}\right) |\phi_0^a|^2 + \left(b^{(1)}-b^{(2)}\right)\overline{\phi_0^a}\phi_0^b \right]K\left|a^{(1)}\phi_0^a+b^{(1)}\phi_0^b\right|^2\right. \\
  \nonumber &+&  \epsilon^2 K\left(\left|a^{(1)}\phi_0^a+b^{(1)}\phi_0^b\right|^2-\left|a^{(2)}\phi_0^a+b^{(2)}\phi_0^b\right|^2\right)\left(a^{(2)}|\phi_0^a|^2 +b^{(2)}\overline{\phi_0^a} \phi_0^b\right) \\
   &+& \left. \left(|a^{(1)}-a^{(2)}|+|b^{(1)}-b^{(2)}|\right) O(\epsilon^4) \right\} d^2\textbf{x}, \\
 \nonumber && E_1^{(1)}\left(b^{(1)}-b^{(2)}\right)+\left(E_1^{(1)}-E_1^{(2)}\right)b^{(2)}\\
  \nonumber &=& \int_{\mathbb{R}^2/\Gamma} \left\{ \epsilon^2\left[\left(b^{(1)}-b^{(2)}\right) |\phi_0^b|^2 + \left(a^{(1)}-a^{(2)}\right)\overline{\phi_0^b}\phi_0^a \right]K\left|a^{(1)}\phi_0^a+b^{(1)}\phi_0^b\right|^2\right. \\
  \nonumber &+&  \epsilon^2 K\left(\left|a^{(1)}\phi_0^a+b^{(1)}\phi_0^b\right|^2-\left|a^{(2)}\phi_0^a+b^{(2)}\phi_0^b\right|^2\right)\left(b^{(2)}|\phi_0^b|^2 +a^{(2)}\overline{\phi_0^b} \phi_0^a\right) \\
   &+& \left. \left(|a^{(1)}-a^{(2)}|+|b^{(1)}-b^{(2)}|\right)O(\epsilon^4) \right\} d^2\textbf{x}.
\end{eqnarray}
Like the arguments in Proposition \ref{Prop-1LS}, we rewrite the above equations in the form of $I_{int}$ and $I_{one}$ as
\begin{eqnarray}
  \nonumber E_1^{(1)}-E_1^{(2)} &=& \frac{a^{(1)}-a^{(2)}}{a^{(2)}}\left(2\epsilon^2|b^{(1)}|^2I_{int}+\epsilon^2|a^{(1)}|^2I_{one}-E_1^{(1)}\right)\\
   &+& \epsilon^2\left(2|b^{(1)}|^2-2|b^{(2)}|\right)I_{int} +\epsilon^2\left(|a^{(1)}|^2-|a^{(2)}|^2\right)I_{one}, \label{P3-coreidentity1} \\
  \nonumber E_1^{(1)}-E_1^{(2)} &=& \frac{b^{(1)}-b^{(2)}}{b^{(2)}}\left(2\epsilon^2|a^{(1)}|^2I_{int}+\epsilon^2|b^{(1)}|^2I_{one}-E_1^{(1)}\right)\\
   &+& \epsilon^2\left(2|a^{(1)}|^2-2|a^{(2)}|\right)I_{int} +\epsilon^2\left(|b^{(1)}|^2-|b^{(2)}|^2\right)I_{one}. \label{P3-coreidentity2}
\end{eqnarray}
Here we complete the proof of the radial separability by classifying them into two cases:

(i) In the case (1) of the proposition, suppose the existences of both two allowed pairs. Remark that the first terms of Eqs. (\ref{P3-coreidentity1}) and (\ref{P3-coreidentity2}) are of order $(|a^{(1)}-a^{(2)}|+|b^{(1)}-b^{(2)}|)/\delta O(\epsilon^4)$ by Proposition \ref{Prop-1LS}. Then $2(2I_{int}-I_{one})(|a^{(1)}|^2-|a^{(2)}|^2)(1+O(\epsilon^2)/{\delta})=0$, which is contradictory when $\epsilon$ is much smaller than $\delta^{1/2}$.

(ii) In the case (2) of the proposition, suppose $a^{(1)}=1$ without loss of generality and the existences of both two allowed pairs. We firstly derive $E_1^{(1)}=\epsilon^2 I_{one} + O(\epsilon^4)$, $E_1^{(2)}= \epsilon^2 I_{one} + O(\epsilon^4)$ from Eq. (\ref{P1-rev-condition-1}) in Proposition \ref{Prop-1LS}. Therefore we have $2(2I_{int}-I_{one})O(\delta)+[2I_{int}-I_{one}+O(\epsilon^2)][1+O(\delta)]=0$. Then we can choose a $\delta$ such that $|O(\delta)|<1/4$, say. At this time, for sufficiently small $\epsilon$ this identity leads to a contradictory.
\end{proof}

It is noteworthy to point out the different preconditions of the two cases. In the first case, as $\epsilon$ goes to $0$, $\delta$ can also tend to $0$, means if there is an allowed pair $(a,b)$ in the first case, then for a long region in the radial direction except the neighborhoods of $(0,1)$ and $(1,0)$, there is no another different one. So it is a powerful tool to exclude the unallowed cases. On the other hand, there is no way from the existence of allowed pairs of the first case to claim the impossibilities of $(0,1)$ and $(1,0)$ being allowed pairs. Therefore, it is necessary to refer to the second case in order to have a complete understanding of the radial separability.

\section{A bootstrap method for the construction of eigenfunctions}\label{sec-main-results}
All the preparation for the proof of the main theorem has been done. In the following passage we focus on the existence of eigenfunctions for certain pairs in $\Sigma_q$. Based on the previous propositions, there are clearly two kinds of potential candidates worth being considered: two "polar" points $(0,1)$, $(1,0)$ and one "equator" $(a,b)$ with $|a|=|b|$. Once they are proved to be in allowed parameter space $\Sigma_q^a$, they are just all the allowed pairs of $\Sigma_q^a$. However, actually things are different in these two cases: While we do construct the corresponding eigenfunctions for the polar points $(0,1)$ and $(1,0)$, we can only obtain a pseudo or approximate eigenfunctions for the generic points of the equator circle. A deeper perturbation analysis gives generally only in 6 points we can get the true eigenfunctions.

Through the demonstration, a bootstrap method will be repeatedly used to construct several convergent series of functions, the limits of which are our targeted eigenfunctions or the pseudo eigenfunctions and thus the existence part of the eigenvalue problem is complete. Fundamentally, this bootstrap method is a rewording of the perturbation theory, which are well-known to and frequently used by physicists. However, a rigorous proof of the existence of eigenfunctions after the so-called perturbation and their regularity will be given in this article.

In order to construct the subsequent reasoning, it is useful to find a complete orthogonal basis of the space $L^2_{\textbf{K}}(\mathbb{R}^2/\Gamma)$. A natural choice is the eigenfunctions of the linear Hamiltonian $H_{L}=-\Delta_\textbf{K}+V_{L}(\textbf{x})$, the completeness of which is guaranteed by the elliptical operator theory. For reference, we denote $(\phi_0^i,E_0^i)$ ($i \in Z^{+}$) as the other eigenfunctions of $H_{L}$ except $\phi_0^a$ and $\phi_0^b$ and their corresponding eigenvalues. Recall that in section \ref{sec-prop-linear} each of these eigenfunctions is also classified into one of the three subspaces: $L^2_{\textbf{K},1}$, $L^2_{\textbf{K},\omega}$ and $L^2_{\textbf{K},\overline{\omega}}$. We denote $i_{1}$, $i_{\omega}$, $i_{\overline{\omega}}$ as the index of the eigenfunctions $\phi_0^{i}$ which lie in the three subspaces, respectively.

\begin{theorem}\label{theorem-main}
Suppose $V_{NL}(\textbf{x},|\psi_i(\textbf{x})|^2)$ in Eq. (\ref{NLS-k}) defined as that in Proposition \ref{Prop-1LS}. If $K(\textbf{x})$ satisfies $I_{one}-2I_{int} \neq 0$, then for sufficiently small $\epsilon>0$, the allowed parameter space $\Gamma_q^a$ defined in section \ref{Sec-uni-and-radial-sepa} satisfies $\{(0,1), \, (1,0)\} \subset \Gamma_q^a \subset \{(0,1), \, (1,0)\} \cup \{(a,b)| |a|= |b|\}$. The corresponding eigenfunctions of the pairs $(0,1)$ and $(1,0)$ are unique, of which the eigenvalues are both equal to $E=E_0+\epsilon^2(I_{one}+O(\epsilon^2))$. Moreover, if the nonlinear term $V_{NL}$ can be further expanded as $V_{NL}=V_L(\textbf{x})+K(\textbf{x})|\phi|^2+M(\textbf{x})|\phi|^4+O(|\phi|^6)$ and below the imaginary part of the complex interaction term is nonzero:
\begin{equation}\label{T1Q-6points-condition}
  I_{c-int} = \left[ 3T_2+\int_{\mathbb{R}^2/\Gamma} M(\overline{\phi_0^a}\phi_0^b)^3 d^2\textbf{x}\right] \neq 0,
\end{equation}
where $T_{2}$ represents
\begin{equation}\label{T1Q-Tree2}
  T_2=-\sum_{n \in i_{1}} \frac{1}{E_0^n-E_0} \int_{\mathbb{R}^2/\Gamma} K \phi_0^b \overline{\phi_0^a}^2 \phi_0^n  d^2\textbf{x} \int_{\mathbb{R}^2} K \overline{\phi_0^a}{\phi_0^b}^2 \overline{\phi_0^n} d^2\textbf{x}.
\end{equation}
then
\begin{eqnarray}
 \nonumber \Gamma_q^a &=& \{(0,1), \, (1,0)\} \cup \left\{ \left(\frac{1}{\sqrt{2}}, \frac{e^{-i\theta_1}}{\sqrt{2}}\right), \, \left(\frac{1}{\sqrt{2}}, \frac{ e^{-i\theta_1+2\pi i /3}}{\sqrt{2}}\right), \, \left(\frac{1}{\sqrt{2}}, \frac{e^{-i\theta_1+4\pi i/3}}{\sqrt{2}}\right) \right\} \\
   && \cup \left\{ \left(\frac{1}{\sqrt{2}}, -\frac{e^{-i\theta_2}}{\sqrt{2}}\right), \, \left(\frac{1}{\sqrt{2}}, -\frac{e^{-i\theta_2+2\pi i /3}}{\sqrt{2}}\right), \, \left(\frac{1}{\sqrt{2}}, -\frac{e^{-i\theta_2+4\pi i/3}}{\sqrt{2}}\right) \right\},
\end{eqnarray}
where $\theta_1,\theta_2 = \arg(I_{c-int})/3+ O(\epsilon^2)$. For each of the eight pairs, there is a unique eigenfunction in $H^2(\mathbb{R}^2/\Gamma)$ and the eigenvalues of the elements in each of the three set are equal. The eigenvalue of the second set is $E=E_0+\epsilon^2(I_{int}+I_{one}/2+O(\epsilon^2))$ and the eigenvalue of the third set has also the same expansion to the order of $O(\epsilon^2)$. All the above eigenfunctions are $C^{\infty}$ functions.
\end{theorem}

\begin{proof}
Based on the aforementioned proposition, we firstly prove the existence of eigenfunctions for the allowed pairs $(0,1)$ and $(1,0)$. Without loss of generality, we only prove the case of $(1,0)$. To begin with, we are trying to find the eigenfunction of Eq. (\ref{LS-series-k}) with $\phi_t=\epsilon \phi_0^a$. As in Proposition. \ref{Prop-1LS}, we can easily get the counterpart in $M_\perp L^2$ space when $\epsilon$ is sufficiently small. We can even acquire the more exact estimation of the operator in Eq. (\ref{P1-solution}). As a result of Lemma \ref{lemma-HSregularity}, we find the operator in Eq. (\ref{P1-solution}) is indeed a contracting mapping in any $H^s$ if $\epsilon$ is supposed to be small enough for each space. We apply the transformed rotation operator $\widetilde{R}$ defined in section \ref{sec-prop-linear} to both sides of Eq. (\ref{P1-solution}). Recall the definition of $\widetilde{R}$, it is readily to find $\widetilde{R} \Delta_{\textbf{K}} \psi = \Delta_{\textbf{K}} \widetilde{R} \psi$. Furthermore, $\widetilde{R}$ is also commutative with $M_\perp$. Then if $|\phi_t|^2$ is invariant under the transformed rotation, $\widetilde{R} \widetilde{\phi}_1 = P(\widetilde{R} \epsilon \phi_0^a)=\omega P\epsilon \phi_0^a=\omega \widetilde{\phi}_1$. This result leads to an miraculous cancellation of Eq. (\ref{P1-rev-condition-2}): the integrations involving $\overline{\phi_0^b}\phi_0^a$ and $\overline{\phi_0^b}\widetilde{\phi}_1$ vanish, so that the second consistency condition is automatically satisfied.

Before going to the bootstrap procedure, we still need to prove the existence of $E_1$ that solve the eigenvalue problem of $H_t$ with $\phi_t=\epsilon \phi_0^a$. Substitute $E_1=(1+ \mu)\epsilon^2 I_{one}$ into Eq. (\ref{P1-rev-condition-1}) and rewrite it as
\begin{equation}\label{T1-contraction}
  \mu = f_{2,1}(\epsilon,\mu),
\end{equation}
where $f_{2,1}$ means $f_{2,1} \sim O (\epsilon^2)$ and $f_{2,1}$ is an Lipschitz function of $\mu$. For $\epsilon$ sufficiently small, $f_{2,1}$ can be defined when $\mu < C$ since we only need to ensure the existence of $P$ in Eq. (\ref{P1-solution}). Suppose $|f_{2,1}|<C|\epsilon|^2$. Since $f_{2,1}(\epsilon,C|\epsilon|^2)<C|\epsilon|^2$, by the Brouwer fixed point theorem, there must exist $\mu_0(\epsilon)\in \mathbb{C}$ such that $|\mu_0(\epsilon)|<C|\epsilon|^2$ and $f_{2,1}(\epsilon,\mu_0(\epsilon))=\mu_0(\epsilon)$. By Proposition \ref{Prop-uniqueness}, $\mu$ and thus the eigenfunction and its corresponding eigenvalue are unique for $E_1$ and $\epsilon$ small enough. Since $H_t$ is a self-adjoint operator, the eigenvalue is also real.

After the construction of the eigenfunction of $H_t$ for $\phi_t=\epsilon \phi_0^a$, next we continue to find an eigenfunction of $H_t$ for $\phi_t=\phi_1=\epsilon \phi_0^a +\widetilde{\phi}_1$. Under a similar procedure, denote the orthogonal counterpart as $\widetilde{\phi}_2$ and the correction of the eigenvalue compared to $E_0$ as $E_2$. For the same reason, we can always find the cancellation of Eq. (\ref{P1-rev-condition-2}) and the existence of $E_2$. The iteration can go on for any $n \in Z^{+}$ and there is a series of $\{(\phi_n,E_n)\}$. To complete the first part of the theorem, we need to prove the convergence of both the two series of $E_n$ and $\phi_n$ through the bootstrap procedure and examine the limit of them turn out to be consistent with Eq. (\ref{NLS-k}).

Similar to the arguments in Lemma \ref{Lemma1-Lpcont}, we estimate $\widetilde{\phi}_{n+1}-\widetilde{\phi}_n$ by Eq. (\ref{P1-2Eqs-1}) and obtain
\begin{eqnarray}
  \nonumber  L\left(\widetilde{\phi}_{n+1}-\widetilde{\phi}_n\right)&=& E_{n+1}\left(\widetilde{\phi}_{n+1}-\widetilde{\phi}_n\right)-M_{\perp}\left[v(|\phi_n|^2)\left(\widetilde{\phi}_{n+1}-\widetilde{\phi}_n\right)\right] \\
   \nonumber &-& M_\perp[ D(|\phi_n|^2,|\phi_{n-1}|^2)(|\phi_n|^2-|\phi_{n-1}|^2)\phi_n] \\
  &+&  (E_{n+1}-E_n)\widetilde{\phi}_n.
\end{eqnarray}
It implies $\|\phi_{n+1}-\phi_n\|_{H^s} \le C\epsilon^2(\|\phi_n -\phi_{n-1}\|_{H^s}+\epsilon |E_{n+1}-E_n|)$. Turning to the series of $E_n$, the estimation of $E_{n+1}-E_{n}$ can be derived by Eq. (\ref{P1-rev-condition-1}) as
\begin{eqnarray}
 \nonumber E_{n+1}-E_n &=&\frac{1}{\epsilon} \int_{\mathbb{R}^2/\Gamma} \left\{v(\phi^{(1)})\left[  (\widetilde{\phi}_{n+1}-\widetilde{\phi}_n)\overline{\phi_0^a} \right] \right.\\
   &+& \left. D(|\phi_n|^2,|\phi_{n-1}|^2)(|\phi_n|^2-|\phi_{n-1}|^2)\overline{\phi_0^a}\phi_{n} \right\} d^2\textbf{x}.
\end{eqnarray}
Therefore $|E_{n+1}-E_n|\le C\epsilon^2 (\|\phi_n-\phi_{n-1}\|_{H^s}+\|\phi_{n+1}-\phi_{n}\|_{H^s})$. Eventually, $\|\phi_{n+1}-\phi_n\|_{H^s} \le C\epsilon^2\|\phi_n-\phi_{n-1}\|_{H^s}$ and $|E_{n+1}-E_n| \le C\epsilon^2\|\phi_n-\phi_{n-1}\|_{H^s}$, showing the convergence of $\{\phi_n\}$ and $E_n$. Denote the limits of these two series are $\phi$ and $E'$ respectively, then
\begin{equation}
 \|H(\textbf{K})\phi-(E_0+E')\phi\|_{H^s} \le C(\|\phi-\phi_n \|_{H^s}+| E'-E_n|),
\end{equation}
the right hand side of which goes to zero as $n\rightarrow \infty$. As a result, $\phi \in H^s(\mathbb{R}^2/\Gamma)$ is the eigenfunction wanted. Since $\Delta_{\textbf{K}} \overline{\phi(-\textbf{x})}= \overline{(\Delta_{\textbf{K}} \cdot \phi)(-\textbf{x})}$, $\overline{\phi(-\textbf{x})}$ is also an eigenfunction of $H_\textbf{K}$ with the same eigenvalue and its projection in $M_{\|}$ is represented by (0,1). The proof of the first case is complete.

Now we consider the case $|a|=|b|$. Given the equivalence relation of the parameter space, assume $|a|=1/\sqrt{2}$ and $|b|=e^{i\beta/\sqrt{2}}$ without loss of generality. Following the similar steps in the first case, we can set $\phi_t=\epsilon (\phi_0^a + e^{i\beta} \phi_0^b)/\sqrt{2}$ and obtain the corresponding $\widetilde{\phi}_1$. Apply the same operator mentioned above, i. e. $\overline{\widetilde{\phi}_1(-\textbf{x})}= (\epsilon/\sqrt{2}) P\overline{(\phi_0^a+e^{i\beta}\phi_0^b)(-\textbf{x})}= (\epsilon/\sqrt{2}) e^{-i\beta}P(\phi_0^a+e^{i\beta}\phi_0^b)=e^{-i\beta}\widetilde{\phi}_1$. Similarly, this identity is also available to any $\phi_t$ such that $\overline{\phi_t(-\textbf{x})}=e^{-i\beta}\phi_t$ if they exist. To continue the same bootstrap procedure, we need to examine the consistency of Eqs. (\ref{P1-rev-condition-1}) and (\ref{P1-rev-condition-2}). Apply the inversion-conjugation operation to both sides of Eq. (\ref{P1-rev-condition-1}) and multiply it by $e^{i\beta}$, we have a similar equation compared to Eq. (\ref{P1-rev-condition-2}):
\begin{equation}
  \int_{\mathbb{R}^2/\Gamma} v(|\phi_t|^2)[e^{i\beta}|\phi_0^b|^2+\overline{\phi_0^b} \phi_0^a +\overline{\phi_0^b} \widetilde{\phi}_1] d^2\textbf{x} = \overline{E_1}e^{i\beta} \phi_0^b.
\end{equation}
For real $E_1$, it implies the automatic consistency of Eq. (\ref{P1-rev-condition-2}). However, in generic cases, $E_1$ can be complex. Although there is no such automatic consistency generally, we can still define a pseudo eigenfunction and its pseudo eigenvalue by simultaneously satisfying the revised Eq. (\ref{P1-solution}) and Eq. (\ref{P1-rev-condition-1}):
\begin{eqnarray}
 \label{T1-pseudo-consistency1} \widetilde{\phi}_1 & = & -\frac{1}{\sqrt{2}}(1+ L^{-1}M_{\perp}[v(|\phi_t|^2)\cdot (\cdot)]-\Re[E_1]L^{-1})^{-1} L^{-1}M_{\perp} [v(|\phi_t|^2)\epsilon(\phi_0^a+e^{i\beta}\phi_0^b)], \\
 \label{T1-pseudo-consistency2}E_1 &=& \int_{\mathbb{R}^2/\Gamma} v(|\phi_t|^2)\left[|\phi_0^a|^2+e^{i\beta}\overline{\phi_0^a} \phi_0^b +\sqrt{2}\frac{\overline{\phi_0^a} \widetilde{\phi}_1}{\epsilon}\right] d^2\textbf{x}.
\end{eqnarray}
From the definition of the solution of the pseudo eigenvalue problem, now the existence of $E_1$ can be proved again from Brouwer fixed point theorem. Actually, suppose $E_1=\epsilon (1+\mu) (I_{one}/2 +I_{int})$. We still have $\mu =f_{2,1}(\epsilon,\mu)$ and $|f_{2,1}| <C\epsilon^2$. For $f_{2,1}$ a Lipschitz function of $\mu$, there exists $\mu_0(\epsilon)$ such that both Eqs. (\ref{T1-pseudo-consistency1}) and (\ref{T1-pseudo-consistency2}) are satisfied by Brouwer fixed point theorem like Eq. (\ref{T1-contraction}). The uniqueness of $E_1$ should be reconsidered. In this situation, suppose they are two pairs $(E_1^{(1)},\phi_1^{(1)})$ and $(E_1^{(2)},\phi_1^{(2)})$. Then
\begin{equation}
  L (\phi_1^{(1)}-\phi_1^{(2)}) + v(|\phi_t|^2)(\phi_1^{(1)}-\phi_1^{(2)}) = \Re(E_1^{(1)}-E_1^{(2)})\phi_1^{(1)} + \Re(E_1^{(2)})(\phi_1^{(1)}-\phi_1^{(2)}).
\end{equation}
Since $|\Re(E_1^{(1)}-E_1^{(2)})|\le |E^{(1)}-E^{(2)}|$, following the same argument in Lemma \ref{Lemma1-Lpcont} can finally prove the uniqueness of $\phi_1$ as a pseudo eigenfunction and $E_1$ as a pseudo eigenvalue for sufficiently small $\epsilon$.

As the bootstrap procedure goes on, we also acquire two series of the eigenpairs $\{\phi_n \}$ and $E_n$ for each $\beta$. Similarly, the limits exist and denote them as  $\phi(\beta)$ and $E'(\beta)$. By the same argument above we can reconstruct Lemma \ref{Lemma1-Lpcont} and Proposition \ref{Prop-uniqueness} even for the pseudo eigenvalue problem, in other words the eigenpairs satisfying Eqs. (\ref{T1-pseudo-consistency1}) and $(\ref{T1-pseudo-consistency2})$ for $\widetilde{\phi}_1$, $\phi_t$ and $E_1$ replaced by $\widetilde{\phi}$, $\phi$ and $E'$. For Proposition \ref{prop-rad-separability}, $E_1^{(1)}$ and $E_1^{(2)}$ of Eq. (\ref{P3-coreidentity2}) should be replaced by $\overline{E_1^{(1)}}$ and $\overline{E_1^{(2)}}$. So at this time we should take the conjugation of Eq. (\ref{P3-coreidentity2}) and calculate the difference of Eqs. (\ref{P3-coreidentity1}) and (\ref{P3-coreidentity2}), the result showing a similar contradiction like in Proposition \ref{prop-rad-separability}.

The remaining work is to figure out the true eigenfunctions from the pseudo ones, meaning the imaginary part of the eigenvalue $\Im(E')=0$. Intuitively, the leading order of the left hand side of Eq. (\ref{T1-pseudo-consistency2}) that has complex contribution may determine the distribution of the allowed pairs in $\Sigma_q^a$ and their corresponding eigenfunctions in the strict meaning. For the order $O(\epsilon^2)$, we have calculate them as $\epsilon^2(I_{one}/2+I_{int})$, which is real and therefore has no effects on the location of the allowed parameter pairs. Before calculate the terms of order $O(\epsilon^4)$, we need to calculate the leading term of $\widetilde{\phi}$. As mentioned before, $\widetilde{\phi}(\beta) \sim O(\epsilon^3)$, so we can rewrite Eq. (\ref{T1-pseudo-consistency1}) as
\begin{equation}
  \widetilde{\phi}(\beta) = -\epsilon^3 L^{-1}M_\perp \left[|\phi_0^a+e^{i\beta}\phi_0^b|^2(\phi_0^a+e^{i\beta}\phi_0^b)\right] + O(\epsilon^5).
\end{equation}
As it states in the beginning of this section, we can regard $M_\perp L^2$ as the closure of the linear span of the basis $\{\phi_0^i\}$. For any $(\phi_0^i, E_0^i)$, we have $L^{-1}\phi_0^i = \phi_0^i/(E_0^i-E_0)$ by definition. Therefore, we can rewrite $\widetilde{\phi}(\beta)$ as
\begin{equation}\label{T1-solution-order3}
  \widetilde{\phi}(\beta) = -\epsilon^3\left[ \sum_{n=1}^{\infty} \frac{\phi_0^n}{E_0^n-E_0} \int_{\mathbb{R}^2/\Gamma} |\phi_0^a+e^{i\beta}\phi_0^b|^2(\phi_0^a+e^{i\beta}\phi_0^b)\overline{\phi_0^n} d^2\textbf{x} \right]+O(\epsilon^5)
\end{equation}
Substituting Eq. (\ref{T1-solution-order3}) into Eq. (\ref{T1-pseudo-consistency2}), The consistency condition for the order $O(\epsilon^4)$ is
\begin{eqnarray}
 \nonumber  && \epsilon^4 \Im \left[ \int_{\mathbb{R}^2/\Gamma}M|\phi_0^a+e^{i\beta}\phi_0^b|^4 e^{i\beta}\overline{\phi_0^a}\phi_0^b  d^2\textbf{x} \right.\\
 \nonumber  &+& \left. \int_{\mathbb{R}^2/\Gamma} \left(2K|\phi_0^a+e^{i\beta}\phi_0^b|^2 \overline{\phi_0^a}\widetilde{\phi}+ K(\phi_0^a+e^{i\beta}\phi_0^b)^2\overline{\widetilde{\phi}}\overline{\phi_0^a}\right)d^2\textbf{x} \right] \\
 \nonumber  &=& \epsilon^4 \Im \left\{ \int_{\mathbb{R}^2/\Gamma}M|\phi_0^a+e^{i\beta}\phi_0^b|^4 e^{i\beta}\overline{\phi_0^a}\phi_0^b  d^2\textbf{x} \right. \\
 \nonumber  &-&   \sum_{n=1}^{\infty} \left[ 2\int_{\mathbb{R}^2/\Gamma} K|\phi_0^a+e^{i\beta}\phi_0^b|^2\overline{\phi_0^a}\phi_0^n d^2\textbf{x} \int_{\mathbb{R}^2/\Gamma} K|\phi_0^a+e^{i\beta}\phi_0^b|^2(\phi_0^a+e^{i\beta}\phi_0^b)\overline{\phi_0^n} d^2\textbf{x}\right. \\
            &+& \left.\left. \int_{\mathbb{R}^2/\Gamma} K|\phi_0^a+e^{i\beta}\phi_0^b|^2(\overline{\phi_0^a}+e^{-i\beta}\overline{\phi_0^b})\phi_0^n d^2\textbf{x} \int_{\mathbb{R}^2/\Gamma} K(\phi_0^a+e^{i\beta}\phi_0^b)^2\overline{\phi_0^a}\overline{\phi_0^n} d^2\textbf{x} \right]\right\}. \label{T1-Energy-order4}
\end{eqnarray}
where $v$ is supposed to be expanded as $v=K(\textbf{x})|\phi|^2+M(\textbf{x})|\phi|^4+O(|\phi|^6)$. To simplify Eq. (\ref{T1-Energy-order4}), we classify the summand in the summation into three categories: $n \in i_{1}$, $n \in i_{\omega}$ and $n \in i_{\overline{\omega}}$, which are defined in the beginning of this section. For $n \in i_{\overline{\omega}}$, since the necessary condition that the integration is nonzero is the integrand remains unchanged under the action of the transformed rotation operator $\widetilde{R}$, the summation of these terms turn out to be real:
\begin{eqnarray}
\nonumber  S_{\overline{\omega}} &=& -\Im \left[ \int 2K\phi_0^n\overline{\phi_0^b}|\phi_0^a|^2 \int K\left(2\phi_0^b\overline{\phi_0^n}|\phi_0^a|^2+\phi_0^b\overline{\phi_0^n}|\phi_0^b|^2\right) \right. \\
   &+& \left. \int K\left(2\phi_0^n\overline{\phi_0^b}|\phi_0^a|^2+\phi_0^n\overline{\phi_0^b}|\phi_0^b|^2\right) \int 2K\phi_0^b\overline{\phi_0^n}|\phi_0^a|^2 \right]=0.
\end{eqnarray}
The same for $n \in i_{\omega}$. Things are different when $n \in i_{1}$. In this case, the imaginary part of the summation can depend on $\beta$:
\begin{eqnarray}
 \nonumber  S_{1} &=& -\Im \left[ \int 2K\overline{\phi_0^a}^2\phi_0^b \phi_0^n \int K\left({\phi_0^a}^2\overline{\phi_0^b}\overline{\phi_0^n}+\overline{\phi_0^a}{\phi_0^b}^2\overline{\phi_0^n}e^{3i\beta}\right)\right.  \\
 \nonumber  &+& \left. \int K\overline{\phi_0^a}{\phi_0^b}^2\overline{\phi_0^n} \int K\left(\phi_0^a\overline{\phi_0^b}^2\phi_0^n + \overline{\phi_0^a}^2\phi_0^b\phi_0^ne^{3i\beta}\right) \right] \\
            &=& -3 \Im\left[e^{3i\beta} \int K\overline{\phi_0^a}^2\phi_0^b\phi_0^n \int K\overline{\phi_0^a} {\phi_0^b}^2 \phi_0^n \right].
\end{eqnarray}
For the first integration of Eq. (\ref{T1-Energy-order4}), following a similar simplification can result in the form of $\Im[ e^{3i\theta}\int  M\overline{\phi_0^a}^3{\phi_0^b}^3]$. Therefore, Eq. (\ref{T1-Energy-order4}) can be rewritten as
\begin{equation}\label{T1-new-consistency}
  \Im E'(\beta) = \Im[\epsilon^4e^{3i\beta}I_{c-int}+O(\epsilon^6)].
\end{equation}
where $I_{c-int}$ is given in Theorem \ref{theorem-main}. For $\beta$ close to $-\theta$ defined in the theorem and sufficiently small $\epsilon$, suppose $\beta=\gamma-\theta$ it is obvious that $\Im E'(\gamma)>0$ for $\pi/3-C\epsilon^2 > \gamma>C\epsilon^2$ and $\Im E'(\gamma)<0$ for $-\pi/3+C\epsilon^2 < \gamma < -C\epsilon^2$. For the zero point theorem, there exist $\beta$ such that $\Im E'=0$. For all the $\beta \in [-\pi,\pi)$, we can conclude that there are at least six points such that let $\Im E'(\beta)=0$, the locations of which are just as described in Theorem \ref{theorem-main}. Furthermore, Applying the transformed rotation operator $\widetilde{R}$ to the eigenfunction $\phi$, we can prove $\widetilde{R}\phi$ and $\widetilde{R}^2\phi$, with $\beta$ transformed to $\beta+2\pi/3$ and $\beta+4\pi/3$, are also the eigenfunctions of $H(\textbf{K})$ of the same eigenvalue because of the commutivity of $\widetilde{R}$ and $H(\textbf{K})$.

The remaining work is the uniqueness of these six allowed pairs. Suppose there are two distinct allowed pairs $(1/\sqrt{2},e^{i\beta^{(1)}}/\sqrt{2})$ and $(1/\sqrt{2},e^{i\beta^{(2)}}/\sqrt{2})$, whose eigenvalues are $E'^{(1)}$ and $E'^{(2)}$, respectively. By Eq. (\ref{L1-eigenvalue-comp}), we find $\Im [E'(\beta^{(1)})-E'(\beta^{(2)})] \sim O (|b^{(1)}-b^{(2)}|)$. Therefore, like Proposition \ref{prop-rad-separability}, we have
\begin{equation}
  \Im(E^{(1)}-E^{(2)}) = \Im\left[(e^{3i\beta^{(1)}}-e^{3i\beta^{(2)}})\left(\epsilon^4 I_{c-int}+\frac{O(\epsilon^6)}{e^{2i\beta^{(1)}}+e^{2i\beta^{(2)}}+e^{i\beta^{(1)}+i\beta^{(2)}}}\right)\right]=0.
\end{equation}
Then suppose one of $\beta^{(i)}$ is not one of the given six points, so for sufficiently small $\epsilon$, we can suppose the two allowed pairs are close enough, like $|\beta^{(1)}-\beta^{(2)}|<\pi/6$, say. Thus the absent value of the denominator of the second term is $|1+e^{i(\beta^{(1)}-\beta^{(2)})}+e^{2i(\beta^{(1)}-\beta^{(2)})}|>C$. For $\epsilon$ sufficiently small,
\begin{eqnarray}
 \nonumber \frac{\arg(E^{(1)}-E^{(2)})}{\epsilon^4} &=& \arg(I_{c-int})+\arg(e^{3i\beta^{(1)}}) +\arg(1-e^{3i(\beta^{(1)}-\beta^{(2)})}) +O(\epsilon^2) \\
   &=& \arg(1-e^{3i(\beta^{(1)}-\beta^{(2)})}) +O(\epsilon^2) \neq 0,
\end{eqnarray}
contradicting to the assumption and we therefore complete the second part of the proof. For the $C^\infty$ continuity of the eigenfunctions, see the appendix.
\end{proof}

\acknowledgments
R. P., Q.F. and F.Y. acknowledge support from NSFC (No.91950120,11690033), Natural Science Foundation of Shanghai  (No.19ZR1424400), and Shanghai Outstanding Academic Leaders Plan (No. 20XD1402000).

\appendix
\section{Analytical analysis of regularity and $C^\infty$ continuity of the eigenfunctions}\label{sec-appen-sobolev}
In the appendix, we add the necessary analysis for the regularity problems occurred in the main text. The primary question is $-(1+ L^{-1}M_{\perp}[v(|\phi_t|^2)\cdot (\cdot)]-E_1L^{-1})^{-1} L^{-1}M_{\perp} [v(|\phi_t|^2)\cdot (\cdot)]$ in Eq. (\ref{P1-solution}) can be defined as a mapping between which spaces. The main difficulties are at the nonlinear term, which needs some techniques to handle the estimation of the norm. As well-known in PDE theory, the elliptic regularities can lift the $L\phi \in H^s$ into $\phi \in H^{s+2}$ for any $s \in Z^{+}\cup \{0\}$. If the operator $P$ can also improve the regularity as expected, one can eventually prove the $C^\infty$ continuity of the eigenfunctions by Sobolev embedding theorem. Indeed, we actually can prove the following lemma for any function $\phi \in H^2$:

\begin{lemma}\label{lemma-HSregularity}
Suppose for sufficiently small $\epsilon>0$, the norm of $\phi_t^{\epsilon}$, which are defined as $H^s(\mathbb{R}^2/\Gamma)$ functions with index $\epsilon$, satisfies $\|\phi_t^\epsilon \|_{H^s} \le C(s)\epsilon$ and $|E_1(\epsilon)| \le C\epsilon$. Then for any $s \in Z^{+}$ ($s \ge 2$), there exist $\epsilon(s)$ such that for any $\epsilon < \epsilon (s)$, the operator $P:H^s \rightarrow H^{s+1}$ such that $P\psi =-(1+ L^{-1}M_{\perp}[v(|\phi_t|^2)\cdot (\cdot)]-E_1L^{-1})^{-1} L^{-1}M_{\perp} [v(|\phi_t|^2)\cdot (\cdot)]$ has the norm of $O(\epsilon^2)$.
\end{lemma}
\begin{proof}
The core step is the regularity of the operator with multiplier $v(|\phi_t|^2)$. According to the Morrey inequality, we have the embedding of $H^s$ into the Holder space $H^s \subset C^{s-2, 1-\delta}$ for any $0<s<1$. Since $v(x)$ is supposed as a $C^{\infty}$ function, $v(|\phi_t|^2)$ is an $C^{s-2}$ function and for any partial derivatives $|D^{i} v(|\phi_t|^2)| \le C\epsilon^2$ in terms of $\textbf{x}$.

We still need to prove it is actually a $H^{s-1}$ function. Recall that by the Sobolev inequality, any $\phi_t$ is also in the Sobolev space $W^{s-1,p}$ for any $1\le p <\infty$. We claim $v(|\phi_t|^2) \in W^{s-1,p}$ for any $p$ and will prove it by definition. We now see $\phi_t$ as defined in $H^s_{loc}(\mathbb{R}^2)$. For the denseness of $C^{\infty}$ functions in any local Sobolev space, we can choose a series $\Phi_j \in C^\infty \cap W^{s-1,p}$ that $\Phi_j \rightarrow \phi_t$. Then it follows that for the partial derivatives of any order $k \le s-1$ $D^{k} v(\textbf{x},\phi)$, where $v$ is considered as the function of $\textbf{x}= \{x,y\}$ and $|\phi|^2$, we have
\begin{equation}\label{LA-ine1}
  \int_{S} |D^{k} v(|\phi_t|^2) -D^{k} v(|\Phi_l|^2)|^p d^2\textbf{x} \le \left(\sup_{|\phi|<|\phi_t|_{L^\infty}} \partial_{\phi} D^{k}v \right)^p\int_S ||\phi_t|^2-|\Phi_l|^2|^p d^2\textbf{x} \rightarrow 0
\end{equation}
for any S a compact subset of $\mathbb{R}^2$. Then we conclude that $D^{k}v(|\Phi_j|^2) \rightarrow D^{k}v(|\phi_t|^2)$ in $L^p_{loc}$ for any $p$. According to the Riesz Theorem, there is an subsequence of $\Phi_j$ such that $\Phi_j \rightarrow \phi_t$ a. e. in $\mathbb{R}^2$. Then from the $C^\infty$ continuity of $D^k v$, $D^k v(|\Phi_l|^2) \rightarrow D^k v(|\phi_t|^2)$ a. e. in $\mathbb{R}^2$. For the formal partial derivatives of $v(|\phi_t|^2)$ in terms of $\textbf{x}$, we can rewrite it as
\begin{equation}\label{LA-def1}
  D^{i}_{fm}v(|\phi_t|^2) = \sum_{0\le k_1,k_2,m_j\le i}^{\sum m_j=i} C\left(m_j, k_1,k_2\right)v_{k_1,k_2}(\textbf{x},\phi_t) \prod_{j=1}^{k_2} D^{m_j} (|\phi_t|^2)
\end{equation}
where all the $D^{i}$ means the partial derivatives of variants $x$ and $y$ and $k_1,k_2$ represent the order of the partial derivative $D^k$ for $\textbf{x}$ and $\phi$, respectively. By the dominated convergence theorem, we can readily conclude that
\begin{eqnarray}
 && \nonumber \int_S |D^{i}v(|\Phi_j|^2)- D^{i}_{fm} v(|\phi_t|^2) |^p d^2\textbf{x} \le C\left(\sup_{|\phi|<|\phi_t|_{L^\infty},k} D^{k}v \right)^p \|\Phi_j-\phi_t \|_{W^{s-1,p}(S)} \\
   &+& C\sum_{0\le k, m_j\le i}^{\sum m_j=i} \int_S |D^k v(|\Phi_j|^2) - D^k v(|\phi_t|^2)|^p \prod_{j=1}^{k_2} D^{m_j} |\phi_t|^{2p} d^2\textbf{x} \rightarrow 0, \label{LA-ine2}
\end{eqnarray}
showing the $L^p$ convergence of $D^{i}v(|\Phi_j|^2) \rightarrow D^i_{fm} v(|\phi_t|^2)$. Recall the definition of the weak derivatives for any function in $H^{s-1,p}$. Then for any testing function $q \in C^\infty_0 (\mathbb{R}^2)$, the identities
\begin{equation}
  \int_S q D^{i}v(|\Phi_j|^2) d^2\textbf{x} = (-1)^{|i|} \int_S v(|\Phi_j|^2) D^{i}q d^2\textbf{x}
\end{equation}
have the limit
\begin{equation}
  \int_S q D^{i}_{fm}v(|\phi_t|^2) d^2\textbf{x} = (-1)^{|i|} \int_S v(|\phi_t|^2) D^{i}q d^2\textbf{x}
\end{equation}
by the convergence of $D^i v(|\Phi_j|^2)$ and $v(|\phi_t|^2)$ in $L^2_{loc}$. Therefore we confirm the claim.

Going through the procedure above for $v[\epsilon^2(|\phi_t|^2/\epsilon^2)]/\epsilon^2$, we have $\|v(|\phi_t|^2)\psi\|_{H^{s-1}} \le  C\epsilon^2\| \psi \|_{H^{s}}$, by choosing $p$ great enough and use the H\"older inequality. $M_\perp$ is a bounded operator for any $H^s$ space, and $L^{-1}$ is a bound mapping from $H^{s-1}$ to $H^{s+1}$ according to the elliptical regularity. Therefore, for $\epsilon$ sufficiently small, $P$ is also well-defined mapping between $H^s \rightarrow H^{s+1}$ and the norm is of the order $O(\epsilon^2)$.
\end{proof}

For the $C^\infty$ continuity of any eigenfunctions of the nonlinear Hamiltonian, one can just prove the similar lemma for $\widetilde{\phi}=L^{-1}M_\perp[E\widetilde{\phi}+v(|\phi|^2)\phi]$. Therefore finish the last part of Theorem \ref{theorem-main}.


\begin{thebibliography}{10}

\bibitem{Graphene-review:RevModPhys.81.109}
A.~H. Castro~Neto, F.~Guinea, N.~M.~R. Peres, K.~S. Novoselov, and A.~K. Geim.
\newblock The electronic properties of graphene.
\newblock {\em Rev. Mod. Phys.}, 81:109--162, Jan 2009.

\bibitem{honeycomb-exp-kekule-PhysRevLett}
Changhua Bao, Hongyun Zhang, Teng Zhang, Xi~Wu, Laipeng Luo, Shaohua Zhou, Qian
  Li, Yanhui Hou, Wei Yao, Liwei Liu, Pu~Yu, Jia Li, Wenhui Duan, Hong Yao,
  Yeliang Wang, and Shuyun Zhou.
\newblock Experimental evidence of chiral symmetry breaking in kekul\'e-ordered
  graphene.
\newblock {\em Phys. Rev. Lett.}, 126:206804, May 2021.

\bibitem{honey-NLS:EXP-Omri}
Omri Bahat-Treidel and Mordechai Segev.
\newblock Nonlinear wave dynamics in honeycomb lattices.
\newblock {\em Phys. Rev. A}, 84:021802, Aug 2011.

\bibitem{Honey-NLS:optical-edgewave-PRA}
Mark~J. Ablowitz, Christopher~W. Curtis, and Yi-Ping Ma.
\newblock Linear and nonlinear traveling edge waves in optical honeycomb
  lattices.
\newblock {\em Phys. Rev. A}, 90:023813, Aug 2014.

\bibitem{SSB:bifurcation-expnature-phon}
Philippe Hamel, Samir Haddadi, Fabrice Raineri, Paul Monnier, Gregoire
  Beaudoin, Isabelle Sagnes, Ariel Levenson, and Alejandro~M. Yacomotti.
\newblock Spontaneous mirror-symmetry breaking in coupled photonic-crystal
  nanolasers.
\newblock {\em Nature Photonics}, 9:311--315, May 2015.

\bibitem{Honey-NLS:Photonic-self-localized-PhysRevLett.111.243905}
Yaakov Lumer, Yonatan Plotnik, Mikael~C. Rechtsman, and Mordechai Segev.
\newblock Self-localized states in photonic topological insulators.
\newblock {\em Phys. Rev. Lett.}, 111:243905, Dec 2013.

\bibitem{Math-Honeycomb:Fefferman}
Charles~L. Fefferman and Michael~I. Weinstein.
\newblock Honeycomb lattice potentials and dirac points.
\newblock {\em J. Amer. Math. Soc.}, 25:1169--1220, Jun 2012.

\bibitem{honey-NLS:Zhuyi-Pra-distorted}
Mark~J. Ablowitz and Yi~Zhu.
\newblock Evolution of bloch-mode envelopes in two-dimensional generalized
  honeycomb lattices.
\newblock {\em Phys. Rev. A}, 82:013840, Jul 2010.

\bibitem{Math-Nonlinear-system-Yang}
Jianke Yang.
\newblock {\em Nonlinear Waves in Integrable and Nonintegrable Systems}.
\newblock Society for Industrial and Applied Mathematics, 2010.

\bibitem{NLS:Smerzi-original-PhysRevLett}
A.~Smerzi, S.~Fantoni, S.~Giovanazzi, and S.~R. Shenoy.
\newblock Quantum coherent atomic tunneling between two trapped bose-einstein
  condensates.
\newblock {\em Phys. Rev. Lett.}, 79:4950--4953, Dec 1997.

\bibitem{NLS:Raghavan-original-PhysRevA}
S.~Raghavan, A.~Smerzi, S.~Fantoni, and S.~R. Shenoy.
\newblock Coherent oscillations between two weakly coupled bose-einstein
  condensates: Josephson effects, $\ensuremath{\pi}$ oscillations, and
  macroscopic quantum self-trapping.
\newblock {\em Phys. Rev. A}, 59:620--633, Jan 1999.

\bibitem{NL:Coullet-non-hermitian-PRE}
P.~Coullet and N.~Vandenberghe.
\newblock Chaotic self-trapping of a weakly irreversible double bose
  condensate.
\newblock {\em Phys. Rev. E}, 64:025202, Jul 2001.

\bibitem{NL:Coullet-non-hermitian-Journalphys-B}
P~Coullet and N~Vandenberghe.
\newblock Chaotic dynamics of a bose-einstein condensate in a double-well trap.
\newblock {\em Journal of Physics B: Atomic, Molecular and Optical Physics},
  35(6):1593--1612, mar 2002.

\bibitem{bifurcation-NL:Rahmi-asym2wells}
Rahmi Rusin, Robert Marangell, and Hadi Susanto.
\newblock Symmetry breaking bifurcations in the nls equation with an asymmetric
  delta potential.
\newblock {\em Nonlinear Dynamics}, 100:3815--3824, Jun 2020.

\bibitem{Math-Bifurcation:periodical-nonlinear-Dohnal}
Tom\'{a}\v{s} Dohnal and Hannes Uecker.
\newblock Bifurcation of nonlinear bloch waves from the spectrum in the
  gross-pitaevskii equation.
\newblock {\em Journal of Nonlinear Science}, 26:581--618, 2016.

\bibitem{NLS:1dimarray-PhysRevLett}
Andrea Trombettoni and Augusto Smerzi.
\newblock Discrete solitons and breathers with dilute bose-einstein
  condensates.
\newblock {\em Phys. Rev. Lett.}, 86:2353--2356, Mar 2001.

\bibitem{Math-NLS:DING-periodical-multisolution}
Yanheng Ding and Cheng Lee.
\newblock Multiple solutions of schr\"odinger equations with indefinite linear
  part and super or asymptotically linear terms.
\newblock {\em Journal of Differential Equations}, 222(1):137--163, 2006.

\bibitem{Math-NLS:Periodical-homoclinic}
Vittorio Coti~Zelati Sissa and Paul~H. Rabinowitz.
\newblock Homoclinic type solutions for a semilinear elliptic pde on rn.
\newblock {\em Communications on Pure and Applied Mathematics},
  45(10):1217--1269, 1992.

\bibitem{Math-NLS:Edge-States-Weinstein-Zhuyi}
J.~P. Lee-Thorp, M.~I. Weinstein, and Y.~Zhu.
\newblock Elliptic operators with honeycomb symmetry: Dirac points, edge states
  and applications to photonic graphene.
\newblock {\em Archive for Rational Mechanics and Analysis}, 232:1--63, Apr
  2019.

\bibitem{honey-NLS:Zhuyi2}
Mark~J. Ablowitz and Yi~Zhu.
\newblock Nonlinear wave packets in deformed honeycomb lattices.
\newblock {\em SIAM Journal on Applied Mathematics}, 73(6):1959--1979, 2013.

\bibitem{Honey-NLS:Zhuyi3}
Mark~J. Ablowitz and Yi~Zhu.
\newblock Nonlinear waves in shallow honeycomb lattices.
\newblock {\em SIAM Journal on Applied Mathematics}, 72(1):240--260, 2012.

\bibitem{Math-NLS:William-localized-equations}
William Borrelli.
\newblock Weakly localized states for nonlinear dirac equations.
\newblock {\em Calculus of Variations and Partial Differential Equations},
  57:155, 2018.

\bibitem{Math-NLS:nonlinear-diracEQ-Jack}
Jack Arbunich and Christof Sparber.
\newblock Rigorous derivation of nonlinear dirac equations for wave propagation
  in honeycomb structures.
\newblock {\em Journal of Mathematical Physics}, 59(1):011509, 2018.

\bibitem{Honey-NLS:LinTai-Chia-Saturable-energy}
Tai-Chia Lin, Milivoj~R. Beli\'{c}, Milan~S. Petrovi\'{c}, and Goong Chen.
\newblock Ground states of nonlinear schr\"odinger systems with saturable
  nonlinearity in r2 for two counterpropagating beams.
\newblock {\em Journal of Mathematical Physics}, 55(1):011505, 2014.

\end{thebibliography}
\end{document}